\documentclass[11pt,dvips,twoside]{article}

\usepackage{pslatex}
\usepackage{fancyhdr}
\usepackage{graphicx}
\usepackage{geometry}

\RequirePackage{amsfonts,amssymb,amsmath,amscd,amsthm}
\RequirePackage{txfonts}
\RequirePackage{graphicx}
\RequirePackage{xcolor}
\RequirePackage{geometry}
\RequirePackage{enumerate}

\def\figurename{Figure} 
\makeatletter
\renewcommand{\fnum@figure}[1]{\figurename~\thefigure.}
\makeatother

\def\tablename{Table} 
\makeatletter
\renewcommand{\fnum@table}[1]{\tablename~\thetable.}
\makeatother

\usepackage{amsmath}
\usepackage{amssymb}
\usepackage{amsfonts}
\usepackage{amsthm,amscd}

\newtheorem{theorem}{Theorem}[section]

\newtheorem{corollary}[theorem]{Corollary}
\newtheorem{proposition}[theorem]{Proposition}

\theoremstyle{example}
\newtheorem{example}[theorem]{Example}

\theoremstyle{definition}
\newtheorem{definition}[theorem]{Definition}

\theoremstyle{remark}

\numberwithin{equation}{section}

\begin{document}

\title{\bfseries\scshape{Classification, $\alpha^k$-Derivations and Centroids of $\textbf{4}$-dimensional Complex Hom-associative Dialgebras}}

\author{\bfseries\scshape \bfseries\scshape Ahmed Zahari\thanks{e-mail address: zaharymaths@gmail.com}\\
Universit\'{e} de Haute Alsace,\\
 IRIMAS-D\'{e}partement de Math\'{e}matiques,\\
18, rue des Fr\`eres Lumi\`ere F-68093 Mulhouse, France.
}

\date{}
\maketitle


\noindent\hrulefill

\noindent {\bf Abstract.}
The basic objective of the current research paper is to investigate the structure and the algebraic varieties of Hom-associative dialgebras. We elaborate a
classification of $n$-dimensional Hom-associative dialgebras for $n\leq4$. Additionally, using the classification result of Hom-associative dialgebras,
we characterize the $\alpha$-derivations and centroids of low-dimensional Hom-associative dialgebras. Furthermore, we equally tackle certain features of derivations and
centroids in the light of associative dialgebras and compute the centroids of low-dimensional associative dialgebras.

\noindent \hrulefill

\vspace{.3in}

 \noindent {\bf AMS Subject Classification: } .

\vspace{.08in} \noindent \textbf{Keywords}:
 Hom-associative dialgebra,  Classification, Derivation, Triple system.
\vspace{.3in}
\vspace{.2in}

\pagestyle{fancy} \fancyhead{} \fancyhead[EC]{ }
\fancyhead[EL,OR]{\thepage} \fancyhead[OC]{Ahmed Zahari} \fancyfoot{}
\renewcommand\headrulewidth{0.5pt}

\emph{
\begin{center}
Dedicated to the memory of professor Marie-Hélène TUILIER, Professor of physics  and  Responsible at the University of Haute-Alsace of doctoral school who has helped
me morally and financially during my doctoral years.
\end{center}
}

\section{Introduction}

A  Hom-diassociative algebra $(\mathcal{A}, \dashv, \vdash,\alpha)$  consists of a vector space, two multiplications, and a linear self map. It may be
regarded as a deformation of an associative algebra, where the associativity condition is twisted by a linear map $\alpha$ , such that
when $\alpha=id$, the Hom-associative dialgebra degenerates to exactly a associative dialgebra.
The central focus of this work is to explore the structure of Hom-associative dialgebras.
 Let $\mathcal{A}$ be an $n$-dimensional $\mathbb{K}$-linear
 vector space and  let $\left\{e_1, e_2, \cdots, e_n\right\}$ be a basis of $\mathcal{A}$, where $\mathbb{K}$ will always be an algebraically closed field of characteristic $0$. A Hom-dialgebra structure on
$\mathcal{A}$ with products $\mu$ and $\lambda$ is determined by
$2n^3$ structure constants $\mathcal{\gamma}_{ij}^k$ and $\mathcal{\delta}_{ij}^k$,  were $e_i\dashv e_j=\sum_{k=1}^n\gamma_{ij}^ke_k,\quad e_i\vdash e_j=\sum_{k=1}^n\delta_{ij}^ke_k$
and by $\alpha$ which is given by ${n^2}$ structure constants $a_{ij}$, where $\alpha(e_i)=\sum_{j=1}^na_{ji}e_j$. Requiring the algebra structure to be Hom-diassociative
and unital  gives rise to a sub-variety $\mathcal{H}d_n$ of $k^{2n^3+n^2}$. Base changes in $A$ result in the natural transport of structure action of $GL_n(k)$ on $\mathcal{H}d_n$. Thus isomorphism classes of $n$-dimensional Hom-dialgebras are in one-to-one correspondence with the orbits of the action of $GL_n(k)$ on $\mathcal{H}d_n$.

In this paper, we tackle the problem of classification. We set forward an algorithm to compute classification. We apply the algorithm in low-dimensional cases. We obtain the classification results of two and three-dimensional complex associative dialgebras from Rikhsiboev
\cite{MRW} and revise a list of four-dimensional nilpotent diassociative algebras from Rakhimov and Fiidov \cite{WRR}. Within this framwork,
A. Zahari and I. Bakayoko studied the classification of BiHom-associative dialgebras \cite{AI}.
The classification of two and three-dimensional Hom-associative dialgebras was undertaken by A. Zahari and A. Makhlouf \cite{AZAB2} and the
classification of $3$-dimensional BiHom-associative and BiHom-bialgebras \cite{AZB} was performed by A. Zahari. Furthermore, we shall consider the class of Hom-associative dialgebras.
We shall also establish a classification of these algebras up to isomorphism in low dimension $n\leq 4$.

The paper is laid out as follows. In the first section, we identify the  basics about Hom-associative dialgebras and provide several new properties.

In section $2$, we address the structure of Hom-associative dialgebras.

Section $3$ is devoted to  the description of the algebraic varieties of Hom-diassociative algebras, and  classifications, up to isomorphism,
of two-dimensional, three-dimensional and four-Hom-associative dialgebras are introduced.

In section $4$,  we determine certain new properties of derivations and we focus upon the classification of the derivations.

Eventually, in Section $5$, we handle the classification of the centroids. In this case, the concept of derivations and centroids is notably inspired  from that of finite-dimensional algebras. The algebra of centroids plays a key role in terms of the classification problems as well as in different applications
of algebras. As far as our work is concerned, we elaborate classification results of two, three and four-dimensional Hom-associative dialgebras.
All considered algebras and vectors spaces  are supposed to be over a field $\mathbb{K}$ of characteristic zero.

\section{Structure of Hom-associative dialgebras}
\begin{definition}\label{dia}
A Hom-associative dialgebra is a $4$-truple $(\mathcal{A}, \dashv, \vdash, \alpha)$ consisting of a  linear space $\mathcal{A}$  linear maps
 $\dashv, \vdash,: \mathcal{A}\times \mathcal{A} \longrightarrow \mathcal{A}$ and  $\alpha : \mathcal{A}\longrightarrow \mathcal{A}$ satisfying, for all $x, y, z\in \mathcal{A}$,
the following conditions :
\begin{eqnarray}
(x\dashv y)\dashv\alpha(z)&=&\alpha(x)\dashv(y\dashv z),\label{eq4}\\
(x\dashv y)\dashv\alpha(z)&=&\alpha(x)\dashv(y\vdash z),\label{eq5}\\
(x\vdash y)\dashv\alpha(z)&=&\alpha(x)\vdash(y\dashv z),\label{eq6}\\
(x\dashv y)\vdash\alpha(z)&=&\alpha(x)\vdash(y\vdash z),\label{eq7}\\
(x\vdash y)\vdash\alpha(z)&=&\alpha(x)\vdash(y\vdash z).\label{eq8}
\end{eqnarray}
\end{definition}
We call $\alpha$ ( in this order ) the structure maps of $\mathcal{A}$.

If in addition, $\alpha$ is an endomorphism with respect to $\dashv$ and $\vdash$, then $\mathcal{A}$ is said to be a multiplicative Hom-dialgebra :
\begin{equation}
\alpha(x\dashv y)=\alpha(x)\dashv\alpha(y)\quad \text{and}\quad\alpha(x\vdash y)=\alpha(x)\vdash\alpha(y)
\end{equation}
for $x, y, z$ in $\mathcal{A}.$

As we are dealing only with multiplicative  Hom-associative dialgebras, we shall call them Hom-diassociative algebras for simplicity. We denote the set of all Hom-associative dialgebras by
$\mathcal{H}d$.
The kernel and the image of homomorphism is defined naturally. One of the basic problems in structure theory of algebras resides in the problem of classification. The classification implies the description of the orbits under base change linear transformations and list representatives of the orbits.
\begin{definition}
Let $(\mathcal{A}_1, \dashv_1, \vdash_1, \alpha_1)$ and $(\mathcal{A}_2, \dashv_2,\vdash_2, \alpha_2)$ be  Hom-associative dialgebras over a field
$\mathbb{K}$. Then, a homomorphism of  Hom-associative dialgebras $\mathcal{A}_1$ to $\mathcal{A}_2$ is a $\mathbb{K}$-linear mapping
$\phi : \mathcal{A}_1\longrightarrow \mathcal{A}_2$ such that
$\phi : \mathcal{A}_1\longrightarrow \mathcal{}_2$ such that
\begin{equation}
\phi(x\dashv_1 y)=\phi(x)\dashv_2 \phi(y),\,\phi(x\vdash_1 y)=\phi(x)\vdash_2\phi(y),\,\phi\circ\alpha_1(x)=\alpha_2\circ\phi(x),
\end{equation}
for all $x,y\in \mathcal{A}.$ A bijective homomorphism is said to be an isomorphism.
\end{definition}

\begin{definition}
A bar unit in $\mathcal{A}$ is an element $e\in \mathcal{A}$ such that $x\dashv e=e\vdash x=\alpha(x)=x.$
\end{definition}

\begin{definition}
A Hom-dendrifom algebra is a quadriple $(\mathcal{A}, \prec, \succ, \alpha)$ consisting of a vector pace $\mathcal{A}$ on which the operations
$\prec, \succ : \mathcal{A}\times \mathcal{A} \longrightarrow \mathcal{A}$ and  $\alpha : \mathcal{A}\longrightarrow \mathcal{A}$ are linear maps satisfying :
\begin{equation}
\begin{aligned}
(x\prec y)\prec\alpha(z)&=\alpha(x)\prec(y\prec z+y \succ z),\\
\alpha(x)\succ(y\succ z)&=(x\prec y+x\succ y)\prec\alpha(z),\\
  (x\succ y)\prec\alpha(z)&=\alpha(x)\succ(y\prec z),
\end{aligned}
\end{equation}
for all $x,y, z\in \mathcal{A}.$
\end{definition}

\begin{example}
In fact, a Hom-dendrifom algebra structure on an n-dimensional vector space $V$ with a basis $\left\{e_1, e_2,\dots, e_n\right\}$ can
be obtained through defining the products and maps of the vectors $\left\{e_1, e_2,\dots, e_n\right\}$. In 2 dimensions, we have :
$$e_2 \prec e_1=e_1,\, e_2 \prec e_2=e_1,\, e_1 \succ e_2=e_1,\,  e_2 \succ e_1=e_1,\, e_2 \succ e_2=e_1,\,\alpha(e_2)=e_1.$$
On a 3-dimensional vector space, define the following Hom-dendrifom algebra structure respectively :
$$e_1 \prec e_1=e_1,\ , e_3 \prec e_2=e_3,\, e_3 \prec e_3=e_3,\, e_2 \succ e_3=e_3,\,  e_3 \succ e_3=e_3,\,\alpha(e_1)=e_1.$$
\end{example}	
\begin{definition}
A Hom-Zinbiel algebra is a triple $(R, \circ, \alpha)$ consisting of vector space A on which $\ast : A\otimes A\longrightarrow A$ and $\alpha : A \longrightarrow A$ are linear maps satisfying
\begin{equation}
(x\circ y)\circ\alpha(z)=\alpha(x)\circ(y\circ z)+\alpha(x)\circ(z\circ y)
\end{equation}
for $x, y, y\in \mathcal{A}.$
\end{definition}
\begin{example}
Let $\left\{e_1,e_2\right\}$ be a basis of $2$-dimensional multiplicative linear space $\mathcal{A}$ over $\mathbb{K}$. The following multiplications and linear map $\alpha$ on
$\mathcal{A}$ are defined by
\begin{enumerate}
\item $e_1\circ e_2=e_1,\quad e_2\circ e_2=e_1,\quad\alpha(e_1)=e_1,\quad \alpha(e_2)=e_2.$
\item $e_1\circ e_1=e_1,\quad e_2\circ e_2=e_1,\quad\alpha(e_1)=e_1,\quad \alpha(e_2)=e_1+e_2.$
	\item $e_1\circ e_1=e_2,\quad e_2\circ e_1=-e_2,\quad\alpha(e_1)=e_1,\quad \alpha(e_2)=e_1+e_2.$
\end{enumerate}
\end{example}	

\begin{proposition}
Let R be a Hom-Zinbiel algebra and put $x \prec y=x\circ y,\quad x\succ y=y\circ x,\,x,y\in R.$ Then $(R, \prec,\succ,\alpha)$ is a Hom-dendriform algebra. Conversely,
a commutative Hom-dendriform algebra $(i.e$ a Hom-dendriform algebra for which $x\prec y=y\succ x)$ is a Hom-dendriform algebra.
\end{proposition}
\begin{proof}
Indeed,
$(i)$\, $(x\prec y)\prec\alpha(z)=(x\circ y)\circ\alpha(z)$.\\
But $\alpha(x)\prec(y\prec z)+\alpha(x)\prec(y\succ y)=\alpha(x)\circ(y\circ z)+\alpha(x)\circ(z\circ y).$ Therefore, (i) holds.\\
$(ii)$\, $(x\succ y)\prec\alpha(z)=(y\circ x)\circ\alpha(z)$ and $\alpha(x)\succ (y\prec z)=(y\circ z)\circ\alpha(x).$ However, these two expressions are the same according to the axioms
of Hom-Zinbiel algebras.\\
$(iii)$\, $(x\prec y)\succ\alpha(z)+(x\succ y)\succ\alpha(z)=\alpha(z)\circ(x\circ y)+\alpha(z)\circ(y\circ x)$, which is equal to $(z\circ y)\circ\alpha(z)=\alpha(x)\succ(y\succ z).$
As a matter of fact, $(iii)$ also holds.
\end{proof}
\begin{proposition}
Let $(R,\cdot, \alpha)$ be a Hom-Zinbiel algebra. Then, the symmetrized product $x\cdot y=x\circ y+y\circ x$ is Hom-associative (,i.e., under the symmetrized product, R becomes a Hom-associative and commutative algebra.
\end{proposition}
\begin{proof}
Indeed,

$\begin{aligned}
(x\cdot y)\alpha(z)&=(x\circ y+y\circ x)\alpha(z)\\
&=(x\circ y+y\circ x)\circ\alpha(z)+\alpha(z)\circ(x\circ y+y\circ x)\circ\alpha(z)\\
&=(x\circ y)\circ\alpha(z)+(y\circ x)\circ\alpha(z)+\alpha(z)\circ(x\circ y)+(y\circ x)\circ\alpha(z)
\end{aligned}$

and

$\begin{aligned}
\alpha(x)(y\cdot z)&=\alpha(x)(y\circ z+z\circ y)\\
&=\alpha(x)\circ(y\circ z+z\circ y)+(y\circ z+z\circ y)\alpha(x)\\
&=\alpha(x)\circ(y\circ z)+\alpha(x)\circ(z\circ y)+(y\circ z)\circ\alpha(x)+(z\circ y)\circ(x).
\end{aligned}$

Now, if we take into account Hom-Zinbiel indentity and its consequence\\
$(y\circ x)\circ\alpha(z)=(y\circ z)\alpha(z)$, then we get the following required equality $(x\cdot y)\alpha(z)=(x\cdot y)\alpha(z).$
\end{proof}

\begin{definition}
A Hom-dipterous algebra is a quadruple $(Z, \prec, \ast, \alpha)$ consisting of a vector space $\mathcal{A}$  with the operations
$\prec, \ast : \mathcal{A}\otimes \mathcal{A}\longrightarrow \mathcal{A}$ and $\alpha : \mathcal{A}\longrightarrow \mathcal{A}$ which are linear maps satisfying
\begin{equation}\label{eq01}
(x\ast y)\succ\alpha(z)=\alpha(x)\succ(y\succ z),\quad (x\ast y)\ast \alpha(z)=\alpha(x)\ast(y\ast z)
\end{equation}
for all $x,y, z\in \mathcal{A}.$
\noindent Similarly, a right Hom-dipterous algebra is defined by the following relations
\begin{equation}\label{eq02}
(x\prec y)\prec\alpha(z)=\alpha(x)\prec(y \ast z),\quad (x\ast y)\ast \alpha(z)=\alpha(x)\ast(y\ast z),
\end{equation}
for all $x,y, z\in \mathcal{A}.$
\end{definition}

\begin{example}
Let $\left\{e_1,e_2,e_3\right\}$ be a basis of $3$-dimensional multiplicative linear space $A$ over $\mathbb{K}$. The following multiplications and linear map
$\alpha$ on $\mathcal{A}$ define the structure of a Hom-dipterous algebra:
\begin{enumerate}
\item $e_1\ast e_2=e_3,\,e_3\ast e_2=e_3,\,e_3\ast e_3=e_3,\quad e_1\prec e_2=e_3,\,e_2\prec e_3=e_3,\,\alpha(e_1)=e_1.$
\item $e_1\ast e_1=e_1,e_1\ast e_2=e_3, e_3\ast e_2=e_3,e_3\ast e_3=e_3,e_1\prec e_2=e_3,e_2\prec e_3=e_3,\, \alpha(e_1)=e_1.$
\end{enumerate}
\end{example}
\begin{proposition}
Let $(\mathcal{A}, \ast, \succ)$ be a dipterous algebra and let $\alpha : \mathcal{A}\longrightarrow \mathcal{A}$ be a dipterous algebra endomorphism. Hence,
$\mathcal{A}_\alpha=(\mathcal{A},\ast_\alpha,\succ_\alpha,\alpha)$, where $\ast_\alpha=\alpha\circ\ast$ and $\succ_\alpha=\alpha\circ\succ$, is a Hom-dipterous algebra.
\end{proposition}
\begin{proof}
Notice that
$$\begin{aligned}
(x\ast_\alpha y)\succ_\alpha\alpha(z)&=\alpha^2((x\ast y)\succ z)\\
\alpha(x)\succ_\alpha(y\succ_\alpha z)&=\alpha^2(x\succ (y\succ z))\\
(x\ast_\alpha y)\ast_\alpha\alpha(z)&=\alpha^2((x\ast y)\ast z)\\
\alpha(x)\ast_\alpha(y\ast_\alpha z)&=\alpha^2(x\ast(y\ast z)).
\end{aligned}$$
Similarly,
$$\begin{aligned}
(x\prec_\alpha y)\prec_\alpha\alpha(z)&=\alpha^2((x\prec y)\prec z)\\
\alpha(x)\prec_\alpha(y\ast_\alpha z)&=\alpha^2(x\prec(y\ast z))\\
(x\ast_\alpha y)\ast_\alpha\alpha(z)&=\alpha^2((x\ast y)\ast z)\\
\alpha(x)\ast_\alpha(y\ast_\alpha z)&=\alpha^2(x\ast(y\ast z)).
\end{aligned}$$
 Thus, the identities (\ref{eq01}) and (\ref{eq02}) obviously follow from the identities satisfied by $(\mathcal{A}, \ast, \succ)$.
\end{proof}

\begin{definition}
Let $(\mathcal{A}, \dashv, \vdash, \alpha)$ be a Hom-associative dialgebra. If there is an associative dialgebra $(A, \dashv', \vdash')$ such that
$\dashv'=\alpha\circ\dashv$ and $\vdash'=\alpha\circ\vdash$, we state that $(\mathcal{A}, \dashv', \vdash')$ is the untwist of $(\mathcal{A}, \dashv, \vdash, \alpha)$.
\end{definition}

\begin{proposition}\label{p11}
Let $(\mathcal{A}, \dashv, \vdash, \alpha)$ be an $n$-dimensional Hom-associative dialgebra and let $\Phi : \mathcal{A}\rightarrow \mathcal{A}$ be an invertible linear map. Then, there is an
isomorphism with an n-dimensional Hom-associative dialgebra  $(\mathcal{A}, \dashv', \vdash', \Phi\alpha\Phi^{-1})$ where
$\dashv'=\Phi\circ\dashv\circ(\Phi^{-1}\otimes\Phi^{-1}$) and $\vdash'=\Phi\circ\vdash\circ(\Phi^{-1}\otimes\Phi^{-1}$).
\end{proposition}

\begin{proof}
We prove that for any invertible linear map $\Phi : \mathcal{A}\rightarrow \mathcal{A}, \, (\mathcal{A}, \dashv', \vdash', \Phi\alpha\Phi^{-1})$ is a Hom-associative dialgebra.
$$\begin{array}{ll}
(x\dashv'y)\dashv'\Phi\alpha\Phi^{-1}(z)
&=(\Phi\circ(\Phi^{-1}(x)\dashv\Phi^{-1}(y)))\dashv'\Phi\circ\alpha\circ\Phi^{-1}(z)\\
&=\Phi\circ(\Phi^{-1}(x)\dashv\Phi^{-1}(y))\dashv\alpha\circ\Phi^{-1}(z)\\
&=\Phi\circ\Phi^{-1}((x)\dashv(y))\dashv\alpha(z)\\
&=\Phi\circ\Phi^{-1}(\alpha(x)\dashv(y))\vdash z)\\
&=\Phi\circ(\alpha\Phi^{-1}(x)\dashv\Phi^{-1}(y))\vdash \Phi^{-1}(z))\\
&=\Phi\circ((\Phi^{-1}\otimes\Phi^{-1})(\Phi\otimes\Phi)\alpha\Phi^{-1}(x)\dashv\Phi^{-1}(y))\vdash \Phi^{-1}(z))\\
&=\Phi\circ(\Phi\circ\alpha\Phi^{-1}(x))\dashv\Phi\circ(\Phi^{-1}(y)\vdash \Phi^{-1}(z))\\
&=\Phi\circ\alpha\Phi^{-1}(x)\dashv'(y\vdash'z).
\end{array}$$

$$\begin{array}{ll}
(x\vdash'y)\dashv'\Phi\alpha\Phi^{-1}(z)
&=\Phi\circ(\Phi^{-1}(x)\vdash\Phi^{-1}(y))\dashv'\Phi\circ\alpha\circ\Phi^{-1}(z)\\
&=\Phi\circ(\Phi^{-1}(x)\vdash\Phi^{-1}(y))\dashv\alpha\circ\Phi^{-1}(z)\\
&=\Phi\circ(\Phi^{-1}((x\vdash y)\dashv\alpha(z)))\\
&=\Phi\circ\Phi^{-1}((\alpha(x)\vdash(y\vdash z)))\\
&=\Phi\circ(\alpha\Phi^{-1}(x)\vdash(\Phi^{-1}(y)\dashv\Phi^{-1}(z)))\\
&=\Phi\circ((\Phi^{-1}\otimes\Phi^{-1})(\Phi\otimes\Phi)\alpha\Phi^{-1}(x)\vdash(\Phi^{-1}(y)\dashv\Phi^{-1}(z)))\\
&=\Phi\circ((\Phi^{-1}\otimes\Phi^{-1})\circ\Phi\alpha\Phi^{-1}(x))\vdash\Phi\circ(\Phi^{-1}(y)\dashv\Phi^{-1}(z)))\\
&=\Phi\circ\alpha\Phi^{-1}(x)\vdash'(y\dashv'z).
\end{array}$$

$$\begin{array}{ll}
(x\vdash'y)\vdash'\Phi\alpha\Phi^{-1}(z)
&=\Phi\circ(\Phi^{-1}(x)\vdash\Phi^{-1}(y))\vdash'\Phi\circ\alpha\circ\Phi^{-1}(z)\\
&=\Phi\circ(\Phi^{-1}(x)\vdash\Phi^{-1}(y))\vdash\alpha\circ\Phi^{-1}(z)\\
&=\Phi\circ(\Phi^{-1}((x\vdash y)\vdash\alpha(z)))\\
&=\Phi\circ\Phi^{-1}((x\dashv y)\vdash z)))\\
&=\Phi\circ(\Phi^{-1}(x)\dashv(\Phi^{-1}y))\vdash\Phi^{-1}(z)))\\
&=\Phi\circ((\Phi^{-1}\otimes\Phi^{-1})(\Phi\otimes\Phi)\circ(\Phi^{-1}(x)\dashv\Phi^{-1}(y))\vdash\Phi^{-1}(z)))\\
&=\Phi\circ((\Phi^{-1}\otimes\Phi^{-1})\circ\Phi(\Phi^{-1}(x)\dashv\Phi^{-1}(y))\vdash\Phi\circ\alpha\circ\Phi^{-1}(z)))\\
&=(x\dashv'y)\vdash\Phi\circ\alpha\Phi^{-1}(z).
\end{array}$$

From this perspective, $(\mathcal{A}, \dashv', \vdash', \Phi\alpha\Phi^{-1})$ is a Hom-associative dialgebra.It is also multiplicative.\\
Indeed, for $\alpha$
$$\begin{array}{ll}
\Phi\alpha\Phi^{-1}\circ(x\dashv'y)
&=\Phi\alpha\Phi^{-1}\Phi\circ x\dashv(\Phi^{-1}\otimes\Phi^{-1})(y)\\
&=\Phi\alpha\circ x\dashv(\Phi^{-1}\otimes\Phi^{-1})(y)\\
&=\Phi\alpha\circ \Phi^{-1}(x)\dashv\Phi^{-1}(y)\\
&=\Phi\circ(\alpha\Phi^{-1}(x)\dashv\alpha\Phi^{-1}(y))\\
&=\Phi\circ(\Phi^{-1}\otimes\Phi^{-1})(\Phi\otimes\Phi)(\alpha\Phi^{-1}(x)\dashv\alpha\Phi^{-1}(y))\\
&=\Phi\circ(\Phi^{-1}\otimes\Phi^{-1})(\Phi\alpha\Phi^{-1}(x)\dashv\alpha\Phi\Phi^{-1}(y))\\
&=\Phi\alpha\Phi^{-1}(x)\dashv'\Phi\alpha\Phi^{-1}(y).
\end{array}$$
\noindent It follows that
$$\begin{array}{ll}
\Phi\alpha\Phi^{-1}\circ(x\vdash'y)
&=\Phi\alpha\Phi^{-1}\Phi\circ x\vdash(\Phi^{-1}\otimes\Phi^{-1})(y)\\
&=\Phi\alpha\circ x\vdash(\Phi^{-1}\otimes\Phi^{-1})(y)\\
&=\Phi\alpha\circ \Phi^{-1}(x)\vdash\Phi^{-1}(y)\\
&=\Phi\circ(\alpha\Phi^{-1}(x)\vdash\alpha\Phi^{-1}(y))\\
&=\Phi\circ(\Phi^{-1}\otimes\Phi^{-1})(\Phi\otimes\Phi)(\alpha\Phi^{-1}(x)\vdash\alpha\Phi^{-1}(y))\\
&=\Phi\circ(\Phi^{-1}\otimes\Phi^{-1})(\Phi\alpha\Phi^{-1}(x)\vdash\alpha\Phi\Phi^{-1}(y))\\
&=\Phi\alpha\Phi^{-1}(x)\vdash'\Phi\alpha\Phi^{-1}(y).
\end{array}$$

\noindent Therefore, $\Lambda : (\mathcal{A}, \dashv, \vdash, \alpha)\rightarrow(\mathcal{A}, \dashv', \vdash',\Phi\alpha\Phi^{-1})$ is a Hom-associative
dialgebras morphism, since\\ $\Phi\circ\dashv=\Phi\circ\dashv\circ(\Phi^{-1}\otimes\Phi^{-1})\circ(\Phi\otimes\Phi)=\dashv'\circ(\Phi\otimes\Phi)$ and
$(\Phi\alpha\Phi^{-1})\circ\Phi=\Phi\circ\alpha.$
\end{proof}

\begin{proposition}
Let $(\mathcal{A}, \dashv, \vdash, \alpha)$ be a Hom-associative dialgebra over $\mathbb{K}$.\\ Let $(\mathcal{A}, \dashv', \vdash', \Phi\alpha\Phi^{-1})$ be
its isomorphic Hom-associative dialgebra described in Proposition \ref{p11}. If $\psi$ is an automorphism of $(\mathcal{A}, \dashv, \vdash, \alpha)$, then
$\Phi\psi\Phi^{-1}$ is an automorphism of $(\mathcal{A}, \dashv, \vdash, \Phi\alpha\Phi^{-1})$.
\end{proposition}

\begin{proof}
Note that $\gamma=\Phi\alpha\Phi^{-1}$. We have
$$\Phi\psi\Phi^{-1}\gamma=\Phi\psi\Phi^{-1}\Phi\alpha\Phi^{-1}=\Phi\psi\alpha\Phi^{-1}=\Phi\alpha\psi\Phi^{-1}=\Phi\alpha\Phi^{-1}\Phi\psi\Phi^{-1}=\gamma\Phi\psi\Phi^{-1}.$$
For any $x,y\in \mathcal{A},$
$$\begin{array}{ll}
\Phi\psi\Phi^{-1}\circ(\Phi(x)\dashv'\Phi(y))
&=\Phi\psi\Phi^{-1}\circ\Phi\circ(x\dashv y)\\
&=\Phi\circ\psi\circ(x\dashv y)\\
&=\Phi\circ(\psi(x)\dashv\psi(y))\\
&=\Phi\circ\psi(x)\dashv'\Phi\circ\psi(y))\\
&=\Phi\psi\Phi^{-1}\Phi(x)\dashv'\Phi\psi\Phi^{-1}\Phi(y)\\
\end{array}$$
This entails,
$$\begin{array}{ll}
\Phi\psi\Phi^{-1}\circ(\Phi(x)\vdash'\Phi(y))
&=\Phi\psi\Phi^{-1}\circ\Phi\circ(x\vdash y)\\
&=\Phi\circ\psi\circ(x\vdash y)\\
&=\Phi\circ(\psi(x)\vdash\psi(y))\\
&=\Phi\circ\psi(x)\vdash'\Phi\circ\psi(y))\\
&=\Phi\psi\Phi^{-1}\Phi(x)\vdash'\Phi\psi\Phi^{-1}\Phi(y)\\
\end{array}$$
 By Definition, $\Phi\psi\Phi^{-1}$ is an automorphism of $(\mathcal{A}, \dashv', \vdash', \Phi\alpha\Phi^{-1})$.
 \end{proof}

\section{Classification in low dimensions}
Let $\mathcal{H}d_n(K)$ denote the variety of $n$ dimensional Hom-associative dialgebras
over a field $\mathcal{A}$. If $\mathcal{A}$ is an $n$-dimensional algebra, then the product of any two elements
$x$ and $y$ can be expressed by the product of basis elements $\left\{e_1,e_2, e_3,\cdots, e_n\right\}$.
Recall that a Hom-diassociative structure on $\mathcal{A}$ can then be defined by two bilinear
mappings :  $\dashv : \mathcal{A}\times \mathcal{A}\longrightarrow \mathcal{A}$ representing the left product,
$\dashv : \mathcal{A} \times \mathcal{A}\longrightarrow \mathcal{A}$ representing the left product $\vdash$ and
$\alpha : \mathcal{A}\longrightarrow \mathcal{A}$ representing a linear map, all satisfying the
 above-mentionend identities when a Hom-associative dialgebra $\mathcal{D}$ can be regarded as a
 quadruplet $\mathcal{D}=(\mathcal{A}, \dashv, \vdash, \alpha)$ where $\dashv, \vdash$ and $\alpha$ are
 Hom-associative dialgebra laws on $\mathcal{A}.$

Let us denote by $\gamma_{ij}^k$, $\delta_{st}^q$ and $a_{ij}$, where
$i, j, k, s, t, q=1, 2, 3, \dots,n$, the structure constants of a Hom-associative dialgebra
with respect to the basis $e_1, e_2, \dots, e_n$ of $\mathcal{A}$.

 \noindent As a result,  $\mathcal{H}d$ can be considered as a closed subset of $2n^3+n^2$-dimensional
 affine space specified by the following system of polynomial equations with respect
 to the structure constants $\gamma_{ij}^k$, $\delta_{st}^q$ and $a_{ji}$ :

 Let $\left\{e_1,e_2, e_3,\cdots, e_n\right\}$ be a basis of an $n$-dimensional Hom-associative dialgebra $\mathcal{A}.$ The product of basis is denoted by

\begin{eqnarray}
\sum_{p=1}^n\sum_{q=1}^n\gamma_{ij}^pa_{qk}\gamma_{pq}^r&=&\sum_{p=1}^n\sum_{q=1}^na_{pi}\gamma_{jk}^q\gamma_{pq}^r,\\
\sum_{p=1}^n\sum_{q=1}^n\gamma_{ij}^pa_{qk}\gamma_{pq}^r&=&\sum_{p=1}^n\sum_{q=1}^na_{pi}\delta_{jk}^q\gamma_{pq}^r=,\\
\sum_{p=1}^n\sum_{q=1}^n\delta_{ij}^pa_{qk}\gamma_{pq}^r&=&\sum_{p=1}^n\sum_{q=1}^na_{pi}\gamma_{jk}^q\delta_{pq}^r,\\
\sum_{p=1}^n\sum_{q=1}^n\gamma_{ij}^pa_{qk}\delta_{pq}^r&=&\sum_{p=1}^n\sum_{q=1}^na_{pi}\gamma_{jk}^q\delta_{pq}^r,\\
\sum_{p=1}^n\sum_{q=1}^n\delta_{ij}^pa_{qk}\delta_{pq}^r&=&\sum_{p=1}^n\sum_{q=1}^na_{pi}\delta_{jk}^q\delta_{pq}^r.
\end{eqnarray}
We seek for all $2$-dimensional Hom-associative dialgebras,  we consider two classes of morphisms which are given by their
 Jordan forms. This implies that they are represented by the matrices
$$
\left(\begin{array}{ccc}
a&0\\
0&b
\end{array}
\right)\quad  \text{and} \quad
\left(\begin{array}{ccc}
a&1\\
0&a
\end{array}
\right).$$
Using similar calculations as depicted in the previous section, we obtain the following classification.
\begin{theorem}\label{Theo1}
Every $2$-dimensional complex Hom-associative dialgebra is isomorphic to one of the following pairwise non-isomorphic
Hom-associative dialgebra  $(\mathcal{A}, \dashv,\vdash,\alpha)$ where $\dashv,\vdash$ are the left product and
right product  and $\alpha$ of the structure map.

$\mathcal{H}d_2^1$ :	$\begin{array}{ll}
e_1\dashv e_1=e_1,\\
e_1\dashv e_2=e_1,\\
\end{array}$
$\begin{array}{ll}
e_2\dashv e_1=e_2,\\
e_2\dashv e_2=e_2,\\
\end{array}$
$\begin{array}{ll}
e_1\vdash e_1=e_1,\\
e_1\vdash e_2=e_2,\\
\end{array}$
$\begin{array}{ll}
e_2\vdash e_1=e_1,\\
e_2\vdash e_2=e_2,\\
\end{array}$
$\begin{array}{ll}
\alpha(e_1)=e_1,\\
\alpha(e_2)=e_2.
\end{array}$\\
	
$\mathcal{H}d_2^2$ :
$\begin{array}{ll}
e_1\dashv e_2=e_2,\,
e_2\dashv e_1=e_1,\,
e_1\vdash e_1=e_1,\,
e_2\vdash e_1=e_1,\,
\end{array}$
$\begin{array}{ll}
\alpha(e_2)=e_1.
\end{array}$\\

$\mathcal{H}d_2^3$ :
$\begin{array}{ll}
e_2\dashv e_2=-\frac{e_1}{2}-e_2,\\
e_2\vdash e_1=e_1,\quad
e_2\vdash e_2=-e_1-e_2,
\end{array}$
$\begin{array}{ll}
\alpha(e_1)=-e_1,\quad
\alpha(e_2)=e_1+e_2.
\end{array}$\\

$\mathcal{H}d_2^4$ :
$\begin{array}{ll}
e_1\dashv e_2=e_1,\\
e_2\dashv e_2=e_1+e_2,\\
\end{array}$
$\begin{array}{ll}
e_1\vdash e_2=e_1,\\
e_2\vdash e_2=e_1+e_2.
\end{array}$
$\begin{array}{ll}
\alpha(e_1)=e_1,\\
\alpha(e_2)=e_1+e_2.
\end{array}$\\

$\mathcal{H}d_2^5$ :
$\begin{array}{ll}
e_1\dashv e_2=e_1,\\
e_2\dashv e_1=e_1,\\
\end{array}$
$\begin{array}{ll}
e_2\dashv e_2=e_1,\quad
e_2\vdash e_1=e_1,\quad
\end{array}$
$\begin{array}{ll}
\alpha(e_2)=e_1.
\end{array}$\\

$\mathcal{H}d_2^6$ :
$\begin{array}{ll}
e_1\dashv e_2=e_1,\\
e_2\dashv e_2=e_2,\\
\end{array}$
$\begin{array}{ll}
e_2\vdash e_1=e_1,\\
e_2\vdash e_2=e_1,\\
\end{array}$
$\begin{array}{ll}
\alpha(e_1)=e_1,\quad
\alpha(e_2)=e_2.
\end{array}$\\

$\mathcal{H}d_2^7$ :
$\begin{array}{ll}
e_1\dashv e_1=e_1+e_2,\,
e_1\vdash e_1=e_1,\,
\end{array}$
$\begin{array}{ll}
e_1\vdash e_2=e_1,\\
\end{array}$
$\begin{array}{ll}
\alpha(e_1)=e_1,\quad
\alpha(e_2)=e_2.
\end{array}$\\

$\mathcal{H}d_2^8$ :
$\begin{array}{ll}
e_1\dashv e_2=e_1,\,
e_2\dashv e_1=e_1,\,
e_2\dashv e_2=e_1,\,
e_2\vdash e_2=e_1,\,
\end{array}$
$\begin{array}{ll}
\alpha(e_2)=e_1.
\end{array}$

$\mathcal{H}d_2^9$ :
$\begin{array}{ll}
e_1\dashv e_2=ae_1,\\
e_2\dashv e_2=be_1+ce_2,\\
\end{array}$
$\begin{array}{ll}
e_2\vdash e_1=fe_1,\\
e_2\vdash e_2=ge_1+ke_2,
\end{array}$
$\begin{array}{ll}
\alpha(e_1)=e_1,\\
\alpha(e_2)=e_1+e_2.
\end{array}$
\end{theorem}

\begin{proof}
Let $\mathcal{A}$ be a two-dimensional vector space. To determine a Hom-associative dialgebra structure on $\mathcal{A}$, we consider $\mathcal{A}$ with respect to 
Hom-diassociative operation.

Let $\mathcal{H}_2^{'}=(\mathcal{A}, \vdash, \alpha)$ be the algebra $$e_1\vdash e_1=e_1,\quad e_1\vdash e_2=e_2,\,\,e_2\vdash e_1=e_1,\,\, e_2\vdash e_2=e_2,\quad \alpha(e_1)=e_1,\quad\alpha(e_2)=e_2.$$
The second multiplication operation $\dashv$ in $\mathcal{A}$, is indicated as follows :
$$e_1\dashv e_1=a_1e_1+a_2,\, e_1\dashv e_2=a_3e_1+a_4e_2,\,e_2\dashv e_1=a_5e_1+a_6e_2,\,e_2\vdash e_2=a_7e_1+a_8e_2.$$

\noindent Now, verifying  Hom-associative dialgebras axioms, we get several constants
for the coefficients $a_i$ where, $i=1,2,\dots,8.$

\noindent Applying
 $(e_1\dashv e_1)\vdash\alpha(e_1)=\alpha(e_1)\vdash(e_1\vdash e_1)$,
 we have $(a_1e_1+a_2e_2)\vdash e_1=e_1\vdash e_1$ . Then, $a_1e_1=e_1.$
 Therefore $a_1=1.$\\

\noindent The verification, $(e_1\vdash e_1)\dashv\alpha(e_1)=\alpha(e_1)\vdash(e_1\dashv e_1)$
yields $(e_1+a_2e_2)\vdash e_1=e_1\vdash e_1$. We have $e_1+a_2e_2=e_1$. Hence we get $a_2=0.$ \\

\noindent Consider $(e_1\dashv e_2)\vdash\alpha(e_1)=\alpha(e_1)\vdash(e_2\vdash e_1)$.
It implies that $(a_3e_1+a_4e_2)\dashv e_1=e_1\dashv e_1$. Hence, $a_3=1$ and $a_4=0.$ \\

\noindent The next relation to consider is $(e_2\vdash e_1)\vdash\alpha(e_2)=\alpha(e_2)\vdash(e_1\dashv e_2)$.
Hence, $a_3=0$ and $a_4=1.$ \\

\noindent Consider $(e_2\dashv e_1)\vdash\alpha(e_2)=\alpha(e_2)\vdash(e_1\vdash e_2)$. It implies that $(a_5e_1+a_6e_2)\vdash e_2)=e_2\vdash e_2 $, Therefore, $a_5=0$ and $a_6=1.$ \\

\noindent Finally, we apply $(e_2\dashv e_2)\vdash\alpha(e_2)=\alpha(e_2)\vdash(e_2\vdash e_2)$. We obtain $(a_7e_1+a_8e_2)\vdash e_2=e_2\vdash e_2$ and we get $a_7=0$ and $a_8=1.$\\
\noindent The verification of all other cases leads to the obtained constraints.\\

\noindent If $a_3=1$ and $a_4=0$, then the right and left product coincide and we get the Hom-associative dialgebra.\\

\noindent If $a_3=0$ and $a_4=1$, we obtain the Hom-associative dialgebras $\mathcal{H}d_2^1$. The other Hom-associative dialgebras
of the list of Theorem \ref{Theo1} can be obtained by minor modification of the  above observation.
\end{proof}

We seek for all $3$-dimensional Hom-associative dialgebras, we consider two classes of morphisms that are provided by their Jordan forms. This implies that they are represented by the matrices
$$
\left(\begin{array}{ccc}
a&0&0\\
0&b&0\\
0&0&c
\end{array}
\right),\quad
\left(\begin{array}{ccc}
a&1&0\\
0&a&0\\
0&0&b
\end{array}
\right),\quad
\left(\begin{array}{ccc}
a&1&0\\
0&a&0\\
0&0&a
\end{array}
\right),\quad
\left(\begin{array}{ccc}
a&1&0\\
0&a&1\\
0&0&a
\end{array}
\right).$$
Using similar calculations as in the previous Section, we obtain the following classification.
\begin{theorem}\label{td2}
Every $3$-dimensional multiplicative  Hom-associative dialgebra is isomorphic to one of the following pairwise non-isomorphic
Hom-associative dialgebras $(\mathcal{A}, \dashv, \vdash, \alpha)$, where $\dashv, \vdash$ are
left product and right product and $\alpha$ of the structure map.

$$\mathcal{H}d_3^1:
\begin{array}{ll}
e_2\dashv e_2=e_1,\\
e_2\dashv e_3=e_1,\\
\end{array}
\begin{array}{ll}
e_3\dashv e_2=e_1,\\
e_3\dashv e_3=e_2,\\
\end{array}
\begin{array}{ll}
e_2\vdash e_2=e_1,\\
e_2\vdash e_3=e_1,\\
\end{array}
\begin{array}{ll}
e_3\vdash e_3=e_1,\\
\alpha(e_2)=e_1.
\end{array}$$

$$\mathcal{H}d_3^2 :
\begin{array}{ll}
e_2\dashv e_1=e_1,\\
e_2\dashv e_3=e_1,\\
\end{array}
\begin{array}{ll}
e_3\dashv e_2=e_1,\\
e_3\dashv e_3=e_2,\\
\end{array}
\begin{array}{ll}
e_2\vdash e_2=e_1,\\
e_2\vdash e_3=e_2,\\
\end{array}
\begin{array}{ll}
e_3\vdash e_3=e_1,\\
\alpha(e_2)=e_1.
\end{array}$$

$$\mathcal{H}d_3^3  :
\begin{array}{ll}
e_2\dashv e_2=e_1,\\
e_2\dashv e_3=e_1,\\
\end{array}
\begin{array}{ll}
e_3\dashv e_2=e_1,\\
e_3\dashv e_3=e_1,\\
\end{array}
\begin{array}{ll}
e_2\vdash e_2=e_1,\\
e_2\vdash e_3=e_1,\\
\end{array}
\begin{array}{ll}
\alpha(e_2)=e_1.
\end{array}$$

$$\mathcal{H}d_3^4 :
\begin{array}{ll}
e_2\dashv e_2=e_1,\\
e_2\dashv e_3=e_1,\\
\end{array}
\begin{array}{ll}
e_3\dashv e_2=e_1,\\
e_3\dashv e_3=e_1,\\
\end{array}
\begin{array}{ll}
e_2\vdash e_2=e_1,\\
e_2\vdash e_3=e_1,\\
\end{array}
\begin{array}{ll}
e_3\vdash e_3=e_1,\\
\alpha(e_1)=e_1.
\end{array}$$

$$\mathcal{H}d_3^5 :
\begin{array}{ll}
e_1\dashv e_1=e_1,\\
e_2\dashv e_2=e_2,\\
\end{array}
\begin{array}{ll}
e_3\dashv e_2=e_2,\\
e_1\vdash e_1=e_1,\\
\end{array}
\begin{array}{ll}
e_2\vdash e_2=e_2,\quad
\alpha(e_1)=e_1.
\end{array}$$

$$\mathcal{H}d_3^6 :
\begin{array}{ll}
e_1\dashv e_2=e_1,\\
e_2\dashv e_1=e_1,\\
e_2\dashv e_2=e_1,\\
\end{array}
\begin{array}{ll}
e_2\dashv e_3=e_1,\\
e_3\dashv e_2=e_3,\\
e_1\vdash e_2=e_1,\\
\end{array}
\begin{array}{ll}
e_2\vdash e_1=e_1,\\
e_2\vdash e_2=e_1,\\
e_2\vdash e_3=e_3,\\
\end{array}
\begin{array}{ll}
e_3\vdash e_2=e_3,\\
e_3\vdash e_3=e_1,\\
\alpha(e_2)=e_1.
\end{array}$$

$$\mathcal{H}d_3^7 :
\begin{array}{ll}
e_2\dashv e_2=e_2,\\
e_2\dashv e_3=e_1,\\
\end{array}
\begin{array}{ll}
e_3\dashv e_3=e_3,\\
e_2\vdash e_2=e_2,\\
\end{array}
\begin{array}{ll}
e_2\vdash e_3=e_1,\\
e_3\vdash e_2=e_1,\\
\end{array}
\begin{array}{ll}
e_3\vdash e_3=e_1,\\
\alpha(e_1)=e_1.
\end{array}$$

$$\mathcal{H}d_3^{8} :
\begin{array}{ll}
e_2\dashv e_2=e_2,\\
e_2\dashv e_3=e_2,\\
\end{array}
\begin{array}{ll}
e_2\vdash e_2=e_2,\\
e_2\vdash e_3=e_2,\\
\end{array}
\begin{array}{ll}
e_3\vdash e_2=e_1,\quad
\alpha(e_1)=e_1.
\end{array}$$

$$\mathcal{H}d_3^{9} :
\begin{array}{ll}
e_2\dashv e_2=e_2,\\
e_2\dashv e_3=e_2,\\
\end{array}
\begin{array}{ll}
e_2\vdash e_2=e_2,\\
e_2\vdash e_3=e_2,\\
\end{array}
\begin{array}{ll}
e_3\vdash e_2=e_1,\\
e_3\vdash e_3=e_3,\\
\end{array}
\begin{array}{ll}
\alpha(e_1)=e_1.
\end{array}$$

$$\mathcal{H}d_3^{10} :
\begin{array}{ll}
e_1\dashv e_2=e_1,\\
e_2\dashv e_1=e_1,\\
e_2\dashv e_3=e_1,\\
\end{array}
\begin{array}{ll}
e_3\dashv e_2=e_3,\\
e_1\vdash e_2=e_1,\\
e_2\vdash e_2=e_1,\\
\end{array}
\begin{array}{ll}
e_2\vdash e_3=e_3,\\
e_3\vdash e_2=e_1,\\
\end{array}
\begin{array}{ll}
\alpha(e_2)=e_1.
\end{array}$$

$$\mathcal{H}d_3^{11} :
\begin{array}{ll}
e_1\dashv e_2=(-1)^{2/3}e_1,\\
e_2\dashv e_1=ae_1,\\
e_2\dashv e_2=be_1+ce_3,
\end{array}
\begin{array}{ll}
e_2\dashv e_3=e_1+de_3,\\
e_3\dashv e_2=e_1+(-1)^{2/3}e_3,\\
e_3\dashv e_3=\frac{i}{\sqrt{3}}e_1,
\end{array}
\begin{array}{ll}
e_1\vdash e_2=(-1)^{2/3}e_1,\\
e_2\vdash e_1=fe_1,\\
e_2\vdash e_2=e_1+g(-1)^{2/3}e_3,
\end{array}
\begin{array}{ll}
e_3\vdash e_2=e_1+e_3,\\
e_3\vdash e_3=\frac{i}{\sqrt{3}}e_1,\\
\alpha(e_2)=e_1.
\end{array}$$

$$\mathcal{H}d_3^{12} :
\begin{array}{ll}
e_1\dashv e_2=e_1,\\
e_2\dashv e_1=e_1,\\
e_2\dashv e_2=e_1+e_3,\\
\end{array}
\begin{array}{ll}
e_2\dashv e_3=e_1+e_3,\\
e_3\dashv e_2=e_1+e_3,\\
e_3\dashv e_3=e_3,\\
\end{array}
\begin{array}{ll}
e_1\vdash e_2=e_1,\\
e_2\vdash e_1=e_1,\\
e_2\vdash e_2=e_1+e_3,\\
\end{array}
\begin{array}{ll}
e_2\vdash e_3=e_1+e_3,\\
e_3\vdash e_2=e_1+e_3,\\
\alpha(e_2)=e_1.
\end{array}$$
$$\mathcal{H}d_3^{13} :
\begin{array}{ll}
e_1\dashv e_1=(-1)^{1/3}e_2,\\
e_1\dashv e_3=ae_2,\\
\end{array}
\begin{array}{ll}
e_3\dashv e_1=be_2,\\
e_3\dashv e_3=\frac{(-1)^{1/3}}{\sqrt{3}}ie_2,\\
\end{array}
\begin{array}{ll}
e_1\vdash e_1=(-1)^{1/3}e_2,\\
e_1\vdash e_3=ce_1,\\
e_3\vdash e_1=\frac{(-1)^{1/3}}{\sqrt{3}}ie_2,\\
\end{array}
\begin{array}{ll}
\alpha(e_1)=e_1,\\
\alpha(e_2)=e_2,\\
\alpha(e_3)=e_3.
\end{array}$$
$$\mathcal{H}d_3^{14} :
\begin{array}{ll}
e_1\dashv e_1=e_1+a(-1)^{2/3}e_3,\\
e_1\dashv e_3=be_3,\\
e_3\dashv e_1=\frac{i}{\sqrt{3}}e_1+c(-1+(-1)^{2/3})e_3,\\
\end{array}
\begin{array}{ll}
e_1\vdash e_1=e_1+d(-1)^{2/3}e_3,\\
e_1\vdash e_3=e_1,\\
e_2\vdash e_2=fe_1+ge_3,\\
\end{array}
\begin{array}{ll}
e_3\vdash e_1=e_1+h(-1)^{2/3}e_3,\\
e_3\vdash e_3=e_1+\frac{i}{\sqrt{3}}e_3,\\
\alpha(e_2)=e_2.
\end{array}$$
\end{theorem}
\begin{proof}
The proof is similar to Theorem \ref{Theo1}
\end{proof}

We seek for all $4$-dimensional Hom-associative dialgebras, we consider two classes of morphisms which are given by
their Jordan forms. This implies that they are represented by the matrices
$$
\left(\begin{array}{cccc}
a&0&0&0\\
0&b&0&0\\
0&0&c&0\\
0&0&0&d
\end{array}
\right),\quad
\left(\begin{array}{cccc}
a&1&0&0\\
0&a&1&0\\
0&0&a&1\\
0&0&0&a
\end{array}
\right),\quad
\left(\begin{array}{cccc}
a&1&0&0\\
0&a&0&0\\
0&0&a&1\\
0&0&0&a
\end{array}
\right),\quad
\left(\begin{array}{cccc}
a&1&0&0\\
0&a&1&0\\
0&0&a&0\\
0&0&0&a
\end{array}
\right),\quad
\left(\begin{array}{cccc}
a&1&0&0\\
0&a&0&0\\
0&0&a&0\\
0&0&0&a
\end{array}
\right).$$
Using similar calculations as displayed in the previous Section, we obtain the following classification.
\begin{theorem}\label{td3}
Every $4$-dimensional multiplicative real Hom-associative dialgebra is isomorphic to one of the following pairwise non-isomorphic
Hom-associative dialgebras $(\mathcal{A}, \dashv, \vdash, \alpha)$, where $\dashv, \vdash$ are
left product and right product and $\alpha$ the structure map.
\end{theorem}
$$\mathcal{H}d_4^1 :\\
\begin{array}{ll}
e_1\dashv e_2=e_1,\\
e_1\dashv e_4=e_3,\\
e_2\dashv e_1=e_1,\\
e_2\dashv e_3=ae_3,\\
\end{array}
\begin{array}{ll}
e_2\dashv e_4=e_1,\\
e_3\dashv e_2=e_1,\\
e_3\dashv e_4=e_1,\\
e_4\dashv e_1=e_3,\\
\end{array}
\begin{array}{ll}
e_4\dashv e_4=e_1,\\
e_1\vdash e_2=e_1,\\
e_1\vdash e_4=e_3,\\
e_2\vdash e_2=e_3,\\
\end{array}
\begin{array}{ll}
e_2\vdash e_3=e_1,\\
e_2\vdash e_4=e_1,\\
e_3\vdash e_2=e_1,\\
e_3\vdash e_4=e_1,\\
\end{array}
\begin{array}{ll}
e_4\vdash e_2=e_1,\\
e_4\vdash e_3=e_1,\\
e_4\vdash e_4=e_1,\\
\alpha(e_2)=e_1.
\end{array}$$

$$\mathcal{H}d_4^2 :\\
\begin{array}{ll}
e_1\dashv e_4=e_3,\\
e_2\dashv e_1=e_1,\\
e_2\dashv e_2=e_3,\\
e_2\dashv e_3=e_1,\\
\end{array}
\begin{array}{ll}
e_2\dashv e_4=e_1,\\
e_3\dashv e_2=e_1,\\
e_3\dashv e_4=e_1,\\
e_4\dashv e_1=e_3,\\
\end{array}
\begin{array}{ll}
e_4\dashv e_4=e_1,\\
e_1\vdash e_2=e_1,\\
e_1\vdash e_4=e_3,\\
e_2\vdash e_2=e_3,\\
\end{array}
\begin{array}{ll}
e_2\vdash e_3=e_1,\\
e_2\vdash e_4=e_1,\\
e_3\vdash e_2=e_1,\\
e_3\vdash e_4=e_1,\\
\end{array}
\begin{array}{ll}
e_4\vdash e_2=e_1,\\
e_4\vdash e_3=e_1,\\
e_4\vdash e_4=e_1,\\
\alpha(e_2)=e_1.
\end{array}$$

$$\mathcal{H}d_4^3 :\\
\begin{array}{ll}
e_1\dashv e_4=e_3,\\
e_2\dashv e_1=e_1,\\
e_2\dashv e_2=e_1+e_3,\\
e_2\dashv e_3=e_1,\\
\end{array}
\begin{array}{ll}
e_2\dashv e_4=e_1,\\
e_3\dashv e_2=e_1,\\
e_3\dashv e_4=e_1,\\
e_4\dashv e_1=e_3,\\
\end{array}
\begin{array}{ll}
e_4\dashv e_4=e_1,\\
e_1\vdash e_2=e_1,\\
e_1\vdash e_4=e_3,\\
e_2\vdash e_2=e_3,\\
\end{array}
\begin{array}{ll}
e_2\vdash e_3=e_1,\\
e_2\vdash e_4=e_3,\\
e_3\vdash e_2=e_1,\\
e_3\vdash e_4=e_1,\\
\end{array}
\begin{array}{ll}
e_4\vdash e_2=e_1,\\
e_4\vdash e_3=e_1,\\
e_4\vdash e_4=e_1,\\
\alpha(e_2)=e_1.
\end{array}$$

$$\mathcal{H}d_4^4 :\\
\begin{array}{ll}
e_2\dashv e_3=e_1+e_3,\\
e_3\dashv e_2=e_1+e_3,\\
e_3\dashv e_4=e_1+e_3,\\
e_4\dashv e_1=e_3,\\
\end{array}
\begin{array}{ll}
e_4\dashv e_3=e_1+e_3,\\
e_4\dashv e_4=e_1+e_3,\\
e_1\vdash e_2=e_1,\\
e_1\vdash e_4=e_3,\\
\end{array}
\begin{array}{ll}
e_2\vdash e_3=e_1+e_3,\\
e_2\vdash e_4=e_1+e_3,\\
e_3\vdash e_2=e_1+e_3,\\
e_3\vdash e_4=e_1+e_3,\\
\end{array}
\begin{array}{ll}
e_4\vdash e_3=e_1+e_3,\\
e_4\vdash e_4=e_1+e_3,\\
\alpha(e_2)=e_1,\\
\alpha(e_4)=e_3.
\end{array}$$

$$\mathcal{H}d_4^5 :\\
\begin{array}{ll}
e_2\dashv e_3=e_3,\\
e_2\dashv e_4=e_1+e_3,\\
e_3\dashv e_2=e_1+e_3,\\
e_3\dashv e_4=e_1+e_3,\\
\end{array}
\begin{array}{ll}
e_4\dashv e_1=e_3,\\
e_4\dashv e_3=e_1+e_3,\\
e_4\dashv e_4=e_1+e_3,\\
e_2\vdash e_1=e_1,\\
\end{array}
\begin{array}{ll}
e_2\vdash e_3=e_1+e_3,\\
e_2\vdash e_4=e_1+e_3,\\
e_3\vdash e_2=e_1+e_3,\\
e_3\vdash e_4=e_1+e_3,\\
\end{array}
\begin{array}{ll}
e_4\vdash e_3=e_1+e_3,\\
e_4\vdash e_4=e_1+e_3,\\
\alpha(e_2)=e_1,\\
\alpha(e_4)=e_3.
\end{array}$$

$$\mathcal{H}d_4^6 :\\
\begin{array}{ll}
e_1\dashv e_2=e_3,\\
e_2\dashv e_3=e_1+e_3,\\
e_2\dashv e_4=e_1+e_3,\\
e_3\dashv e_2=e_1+e_3,\\
\end{array}
\begin{array}{ll}
e_3\dashv e_4=e_1,\\
e_4\dashv e_2=e_3,\\
e_4\dashv e_3=e_1+e_3,\\
e_4\dashv e_4=e_1+e_3,\\
\end{array}
\begin{array}{ll}
e_2\vdash e_1=e_1,\\
e_2\vdash e_3=e_1+e_3,\\
e_2\vdash e_4=e_1+e_3,\\
e_3\vdash e_2=e_1+e_3,\\
\end{array}
\begin{array}{ll}
e_3\vdash e_4=e_1+e_3,\\
e_4\vdash e_3=e_1+e_3,\\
\alpha(e_2)=e_1,\\
\alpha(e_4)=e_3.
\end{array}$$

$$\mathcal{H}d_4^7 :\\
\begin{array}{ll}
e_2\dashv e_3=e_4,\\
e_3\dashv e_2=e_4,\\
e_3\dashv e_4=e_4,\\
\end{array}
\begin{array}{ll}
e_4\dashv e_3=e_4,\\
e_4\dashv e_4=e_2+e_4,\\
e_3\vdash e_2=e_2,\\
\end{array}
\begin{array}{ll}
e_3\vdash e_4=e_2+e_4,\\
e_4\vdash e_3=e_2+e_4,\\
e_4\vdash e_4=e_2+e_4,\\
\end{array}
\begin{array}{ll}
\alpha(e_1)=e_1,\\
\alpha(e_2)=e_2.
\end{array}$$

$$\mathcal{H}d_4^8 :\\
\begin{array}{ll}
e_2\dashv e_3=e_4,\\
e_3\dashv e_2=e_4,\\
e_3\dashv e_3=e_4,\\
\end{array}
\begin{array}{ll}
e_3\dashv e_4=e_4,\\
e_4\dashv e_3=e_4,\\
e_4\dashv e_4=e_4,\\
\end{array}
\begin{array}{ll}
e_3\vdash e_2=e_2,\\
e_3\vdash e_4=e_2+e_4,\\
e_4\vdash e_3=e_2+e_4,\\
\end{array}
\begin{array}{ll}
e_4\vdash e_4=e_2+e_4,\\
\alpha(e_1)=e_1,\\
\alpha(e_2)=e_2.
\end{array}$$

$$\mathcal{H}d_4^{9} :\\
\begin{array}{ll}
e_1\dashv e_1=e_1,\\
e_2\dashv e_4=e_1,\\
e_3\dashv e_3=e_2,\\
\end{array}
\begin{array}{ll}
e_3\dashv e_4=e_1,\\
e_4\dashv e_1=e_1,\\
e_4\dashv e_2=e_1,\\
\end{array}
\begin{array}{ll}
e_4\dashv e_3=e_1,\\
e_4\dashv e_4=e_1,\\
e_2\vdash e_4=e_2,\\
\end{array}
\begin{array}{ll}
e_3\vdash e_3=e_2,\\
e_3\vdash e_4=e_1+e_2,\\
e_4\vdash e_3=e_1+e_2,\\
\end{array}
\begin{array}{ll}
e_4\vdash e_4=e_2,\\
\alpha(e_2)=e_2,\\
\alpha(e_3)=e_3.
\end{array}$$

$$\mathcal{H}d_4^{10} :\\
\begin{array}{ll}
e_2\dashv e_2=e_2+e_4,\\
e_2\dashv e_4=e_2+e_4,\\
e_3\dashv e_3=e_1,\\
\end{array}
\begin{array}{ll}
e_4\dashv e_4=e_4,\\
e_2\vdash e_2=e_2+e_4,\\
e_2\dashv e_4=e_2+e_4,\\
\end{array}
\begin{array}{ll}
e_3\vdash e_3=e_2+e_4,\\
e_4\vdash e_2=e_4,\\
e_4\vdash e_4=e_2+e_4,\\
\end{array}
\begin{array}{ll}
\alpha(e_1)=e_1,\\
\alpha(e_3)=e_3.
\end{array}$$

$$\mathcal{H}d_4^{11} :\\
\begin{array}{ll}
e_1\dashv e_2=e_1,\\
e_2\dashv e_1=e_1,\\
e_2\dashv e_2=e_1,\\
\end{array}
\begin{array}{ll}
e_2\dashv e_3=e_1,\\
e_3\dashv e_2=e_1,\\
e_4\dashv e_2=e_1,\\
\end{array}
\begin{array}{ll}
e_4\dashv e_4=e_3,\\
e_1\vdash e_2=e_1,\\
e_2\vdash e_2=e_1,\\
\end{array}
\begin{array}{ll}
e_2\vdash e_3=e_1,\\
e_2\vdash e_4=e_1+e_3,\\
e_4\vdash e_2=e_1+e_3,\\
\end{array}
\begin{array}{ll}
e_4\vdash e_4=e_1,\\
\alpha(e_3)=e_3,\\
\alpha(e_4)=e_4.
\end{array}$$

$$\mathcal{H}d_4^{12} :\\
\begin{array}{ll}
e_2\dashv e_2=e_2+e_4,\\
e_2\dashv e_4=e_2+e_4,\\
e_3\dashv e_3=e_1,\\
\end{array}
\begin{array}{ll}
e_4\dashv e_2=e_2,\\
e_4\dashv e_4=e_4,\\
e_2\vdash e_2=e_2+e_4,\\
\end{array}
\begin{array}{ll}
e_2\dashv e_4=e_2+e_4,\\
e_3\vdash e_3=e_2+e_4,\\
e_4\vdash e_2=e_4,\\
\end{array}
\begin{array}{ll}
e_4\vdash e_4=e_2+e_4,\\
\alpha(e_1)=e_1,\\
\alpha(e_3)=e_3.
\end{array}$$

$$\mathcal{H}d_4^{13} :\\
\begin{array}{ll}
e_2\dashv e_2=e_2,\\
e_2\dashv e_3=e_2+e_3,\\
e_3\dashv e_3=e_2,\\
\end{array}
\begin{array}{ll}
e_4\dashv e_4=e_1,\\
e_2\vdash e_2=e_1+e_3,\\
e_2\vdash e_3=e_2,\\
\end{array}
\begin{array}{ll}
e_3\dashv e_2=e_1+e_3,\\
e_3\vdash e_3=e_1+e_3,\\
e_4\vdash e_4=e_1+e_2,\\
\end{array}
\begin{array}{ll}
\alpha(e_1)=e_1,\\
\alpha(e_4)=e_4.
\end{array}$$

$$\mathcal{H}d_4^{14} :\\
\begin{array}{ll}
e_2\dashv e_2=e_2,\\
e_2\dashv e_3=e_2+e_3,\\
e_3\dashv e_3=e_2+e_3,\\
\end{array}
\begin{array}{ll}
e_4\dashv e_4=e_1,\\
e_2\vdash e_2=e_1+e_3,\\
e_2\vdash e_3=e_1+e_2,\\
\end{array}
\begin{array}{ll}
e_3\dashv e_2=e_1+e_3,\\
e_3\vdash e_3=e_1+e_3,\\
e_4\vdash e_4=e_1+e_2,\\
\end{array}
\begin{array}{ll}
\alpha(e_1)=e_1,\\
\alpha(e_4)=e_4.
\end{array}$$

$$\mathcal{H}d_4^{15} :\\
\begin{array}{ll}
e_2\dashv e_4=e_1,\\
e_3\dashv e_3=e_2,\\
e_3\dashv e_4=e_1,\\
\end{array}
\begin{array}{ll}
e_4\dashv e_1=e_1,\\
e_4\dashv e_2=e_1,\\
e_4\dashv e_3=e_1,\\
\end{array}
\begin{array}{ll}
e_4\dashv e_4=e_1,\\
e_2\vdash e_4=e_2,\\
e_3\vdash e_3=e_1+e_2,\\
\end{array}
\begin{array}{ll}
e_3\vdash e_4=e_1+e_2,\\
e_4\vdash e_1=e_1+e_2,\\
e_4\vdash e_2=e_1+e_2,\\
\end{array}
\begin{array}{ll}
e_4\vdash e_3=e_1+e_2,\\
\alpha(e_2)=e_2,\\
\alpha(e_3)=e_3.
\end{array}$$

$$\mathcal{H}d_4^{16} :\\
\begin{array}{ll}
e_2\dashv e_4=e_1,\\
e_3\dashv e_3=e_2,\\
e_3\dashv e_4=e_1,\\
\end{array}
\begin{array}{ll}
e_4\dashv e_1=e_1,\\
e_4\dashv e_2=e_1,\\
e_4\dashv e_3=e_1,\\
\end{array}
\begin{array}{ll}
e_4\dashv e_4=e_1,\\
e_2\vdash e_4=e_1+e_2,\\
e_3\vdash e_3=e_2,\\
\end{array}
\begin{array}{ll}
e_3\vdash e_4=e_1+e_2,\\
e_4\vdash e_1=e_1+e_2,\\
e_4\vdash e_2=e_1+e_2,\\
\end{array}
\begin{array}{ll}
e_4\vdash e_4=e_2,\\
\alpha(e_2)=e_2,\\
\alpha(e_3)=e_3.
\end{array}$$
$$\mathcal{H}d_4^{17} :\\
\begin{array}{ll}
e_1\dashv e_4=e_1,\\
e_2\dashv e_4=e_1,\\
e_3\dashv e_3=e_2,\\
\end{array}
\begin{array}{ll}
e_3\dashv e_4=e_1,\\
e_4\dashv e_1=e_1,\\
e_4\dashv e_2=e_1,\\
\end{array}
\begin{array}{ll}
e_4\dashv e_3=e_1,\\
e_4\dashv e_4=e_1,\\
e_2\vdash e_4=e_1+e_2,\\
\end{array}
\begin{array}{ll}
e_3\vdash e_3=e_1+e_2,\\
e_4\vdash e_2=e_2,\\
e_4\vdash e_3=e_1+e_2,\\
\end{array}
\begin{array}{ll}
e_4\vdash e_4=e_2,\\
\alpha(e_2)=e_2,\\
\alpha(e_3)=e_3.
\end{array}$$

$$\mathcal{H}d_4^{18} :\\
\begin{array}{ll}
e_1\dashv e_2=e_3+e_4,\\
e_2\dashv e_1=e_3+e_4,\\
e_2\dashv e_2=e_3+e_4,\\
\end{array}
\begin{array}{ll}
e_2\dashv e_3=e_3+e_4,\\
e_2\dashv e_4=e_3+e_4,\\
e_3\dashv e_2=e_3+e_4,\\
\end{array}
\begin{array}{ll}
e_3\dashv e_3=e_3+e_4,\\
e_3\dashv e_4=e_3+e_4,\\
e_4\dashv e_2=e_3+e_4,\\
\end{array}
\begin{array}{ll}
e_1\vdash e_1=e_3+e_4,\\
e_2\vdash e_2=e_3+e_4,\\
e_2\vdash e_3=e_3+e_4,\\
\end{array}
\begin{array}{ll}
e_3\vdash e_2=e_3+e_4,\\
e_4\vdash e_4=e_3+e_4,\\
\alpha(e_2)=e_1.
\end{array}$$

$$\mathcal{H}d_4^{19} :\\
\begin{array}{ll}
e_1\dashv e_2=ae_3+be_4,\\
e_2\dashv e_1=e_3+ce_4,\\
e_2\dashv e_2=e_3+de_4,\\
\end{array}
\begin{array}{ll}
e_2\dashv fe_3=e_3+e_4,\\
e_2\dashv e_4=he_3+e_4,\\
e_3\dashv e_2=le_3+e_4,\\
\end{array}
\begin{array}{ll}
e_3\dashv e_3=e_3+e_4,\\
e_3\dashv e_4=e_3+e_4,\\
e_4\dashv e_2=e_3+e_4,\\
\end{array}
\begin{array}{ll}
e_1\vdash e_1=me_3+ne_4,\\
e_2\vdash e_2=oe_3+pe_4,\\
e_2\vdash e_3=ce_3+de_4,\\
\end{array}
\begin{array}{ll}
e_3\vdash e_2=e_3,\\
e_4\vdash e_4=e_4,\\
\alpha(e_2)=e_1.
\end{array}$$
$$\mathcal{H}d_4^{20} :\\
\begin{array}{ll}
e_2\dashv e_1=e_3+ae_4,\\
e_2\dashv e_2=e_3+be_4,\\
e_2\dashv e_3=ce_3+de_4,\\
\end{array}
\begin{array}{ll}
e_2\dashv e_4=e_3+e_4,\\
e_3\dashv e_2=fe_3+ge_4,\\
e_3\dashv e_3=e_3+e_4,\\
\end{array}
\begin{array}{ll}
e_3\dashv e_4=e_3+e_4,\\
e_4\dashv e_2=e_3+e_4,\\
e_1\vdash e_1=e_3+e_4,\\
\end{array}
\begin{array}{ll}
e_2\vdash e_2=he_3+e_4,\\
e_2\vdash e_3=ke_3+e_4,\\
e_3\vdash e_2=le_3,\\
\end{array}
\begin{array}{ll}
e_4\vdash e_3=e_3,\\
e_4\vdash e_4=e_4,\\
\alpha(e_2)=e_1.
\end{array}$$
\begin{proof}
The proof is similar to Theorem \ref{Theo1}
\end{proof}

 \section{Derivations of Complex Hom-associative dialgebras}
 This section is notably devoted to the description of derivations of two, three and four-dimensional Complex Hom-associative dialgebras.

\subsection{Derivations of Complex Hom-associative dialgebras}
 Let $(\mathcal{A}, \dashv,\vdash  , \alpha)$ be a multiplicative Hom-associative dialgebra. For any nonnegative integer $k$,
 we denote by $\alpha^{k}$  the $k$-fold  composition of $\alpha$ with itself,i.e.,
 $\alpha^k=\alpha\circ\cdots\circ\alpha$ ($k$-times).  In particular, $\alpha^0=id$ and $\alpha^1=\alpha$.
\begin{definition}\label{df6}
For any non-negative integer $k$, a linear map $D : \mathcal{A}\longrightarrow \mathcal{A}$ is called an $\alpha^k-$derivation of a
Hom-associative dialgebra $(\mathcal{A}, \dashv,\vdash  , \alpha)$, if
\begin{equation}\label{eq6}
  D\circ\alpha=\alpha\circ D
\end{equation}
\begin{equation}
D\circ (f\dashv  g)=(D(f)\dashv \alpha^k(g))+(\alpha^k(f) \dashv D(g)).
\end{equation}

\begin{equation}
D\circ (f\vdash g)=(D(f)\vdash \alpha^k(g))+(\alpha^k(f) \vdash D(g)).
\end{equation}
\end{definition}

The map $D(f)$ is an $\alpha^{k+1}$-derivation, which we will call an
\textbf{inner $\alpha^{k+1}$}-derivation. In fact, we
have $D_k(f)(\alpha(g))=\alpha^{k+1}(g)\dashv f =\alpha(\alpha^k(g)
\dashv f)=\alpha\circ D_k(f)(g)$,\\
$D_k(f)(\alpha(g))=\alpha^{k+1}(g)\vdash f =\alpha(\alpha^k(g)
\vdash f)=\alpha\circ D_k(f)(g)$,

which implies that identity  (\ref{eq6})
in Definition \ref{df6} is satisfied. On the other side, we have
$$\begin{array}{ll}
D_k(f)(g\dashv h)
&=\alpha^k((g \dashv h)\dashv f)=(\alpha^k(g)\dashv\alpha^k(h))\dashv\alpha(f))\\
&=\alpha^{k+1}(g)\dashv\alpha^{k}(h))\dashv f+(\alpha^k(g)\dashv f)\dashv \alpha^{k+1}(h)\\
&=(\alpha^{k+1}(g)\dashv D_k(f)(h))+D_k(f)(g)\dashv\alpha^{k+1}(h).
\end{array}$$
Therefore, $D_k(f)$ is an $\alpha^{k+1}$-derivation.  The set of $\alpha^k$-derivations denoted by \textbf{Inner$_{\alpha^k}(A)$} are expressed in terms of
\begin{equation}
\textbf{Inner}_{\alpha^k}(A)=\left\{\alpha^{k-1}(\bullet)\dashv f|f\in A, \alpha(f)=f\right\}.
\end{equation}
\begin{equation}
\textbf{Inner}_{\alpha^k}(A)=\left\{\alpha^{k-1}(\bullet)\vdash f|f\in A, \alpha(f)=f\right\}.
\end{equation}
For any $D\in Der _{\alpha^k}(A)$ and $D'\in Der _{\alpha^s}(A)$, we define their commutator $\left[D, D'\right]$ as usual :
$\left[D,D'\right]=D\circ D'-D'\circ D$.
\begin{proposition}
For any $D\in Der_{\alpha^k}(A)$ and  $D'\in Der_{\alpha^s}(A)$, we have
$\left[D,D'\right]\in Der_{\alpha^{k+s}}(A).$
\end{proposition}

 \begin{proof}
For any $f, g\in \mathcal{A},$ we have
$$\begin{array}{ll}
& \left[D,D'\right](f\dashv g)
=D\circ D'(f\dashv g)-D'\circ D(f\dashv g)\\
&=D(D'(f)\dashv\alpha^s(g))+\alpha^s(f)\dashv D'(g)))
-D'(\mu(D(f),\alpha^k(g))+\mu(\alpha^k(f),D(g)))\\
&=D\circ D'(f)\dashv\alpha^{k+s}(g))+\alpha^k\circ D'(f)\dashv D\circ\alpha^s(g)
+D\circ\alpha^s(f)\dashv\alpha^k\circ D'(g)\\
&+\alpha^{k+s}(f)\dashv D\circ D'(g))-D'\circ D(f)\dashv\alpha^{k+s}(g))-\alpha^s\circ D(f)\dashv D'\circ\alpha^k(g)\\
&-D'\circ\alpha^k(f)\dashv\alpha^s\circ D(g)-\alpha^{k+s}(f)\dashv D'\circ D(g).
\end{array}$$
Since $D$ and $D'$ satisfy $D\circ\alpha=\alpha\circ D,\quad D'\circ\alpha=\alpha\circ D'$, we obtain \\
$\alpha^k\circ D'=D'\circ\alpha^k, \quad D\circ\alpha^s=\alpha^s\circ D$. Therefore, we get
\begin{equation}
\left[D,D'\right](f\dashv g)=\alpha^{k+s}(f)\dashv\left[D,D'\right](g))+\left[D,D'\right](f)\dashv \alpha^{k+s}(g)).\nonumber
\end{equation}
Furthermore, it is straightforward to infer that
$$
\left[D,D'\right]\circ\alpha=D\circ D'\circ\alpha-D'\circ D\circ\alpha
=\alpha\circ D\circ D'-\alpha\circ D'\circ D
=\alpha\circ\left[D, D'\right],
$$
which yields that $\left[D, D'\right]\in Der_{\alpha^{k+s}}(\mathcal{A})$.
\end{proof}

\begin{definition}
A Hom-associative triple system is a vector space $A$ over field $\mathbb{K}$ with a trilinear multiplications satisfying
$$\begin{array}{ll}
(((x\dashv y)\dashv\alpha(z))\dashv\alpha(u))\dashv\alpha(w)
&=((\alpha(x)\dashv(y\dashv z))\dashv\alpha(u))\dashv\alpha(w))\\
&=\alpha(x)\dashv(\alpha(y)\dashv(z\dashv u)\dashv\alpha(w))),
\end{array}$$
$$\begin{array}{ll}
(((x\vdash y)\vdash\alpha(z))\vdash\alpha(u))\vdash\alpha(w)
&=((\alpha(x)\vdash(y\vdash z))\vdash\alpha(u))\vdash\alpha(w))\\
&=\alpha(x)\vdash(\alpha(y)\vdash(z\vdash u)\vdash\alpha(w))),
\end{array}$$
for any $x,y,z,u,w\in \mathcal{A}$.
\end{definition}

\begin{definition}
An  associative triple derivation of Hom-associative dialgebra $(A, \dashv, \vdash, \alpha)$ is a linear transformation $D : \mathcal{A}\longrightarrow \mathcal{A}$
such that $$D\circ\alpha=\alpha\circ D$$
$$\begin{array}{ll}
D\circ((x\dashv y)\dashv z))
&=(D(x)\dashv\alpha^k(y))\dashv\alpha^k(z)\\
&+(\alpha^k(x)\dashv D(y))\dashv\alpha^k(z)+(\alpha^k(x)\dashv\alpha^k(y))\dashv D(z)
\end{array}$$
$$\begin{array}{ll}
D\circ((x\vdash y)\vdash z))
&=(D(x)\vdash\alpha^k(y))\vdash\alpha^k(z)\\
&+(\alpha^k(x)\vdash D(y))\vdash\alpha^k(z)+(\alpha^k(x)\vdash\alpha^k(y))\vdash D(z)
\end{array}$$
for $x, y, z\in \mathcal{A}.$
\end{definition}

\begin{definition}
A Jordan associative triple derivation of Hom-associative $(\mathcal{A}, \dashv, \vdash \alpha)$ is a linear
transformation $D' : \mathcal{A}\longrightarrow A$ such that $$D'\circ\alpha=\alpha\circ D'$$ and
$$\begin{array}{ll}
D'\circ((x\dashv y)\dashv x))
&=(D'(x)\dashv\alpha^k(y))\dashv\alpha^k(x)\\
&+(\alpha^k(x)\dashv D'(y))\dashv\alpha^k(x)+(\alpha^k(x)\dashv\alpha^k(y))\dashv D'(x)
\end{array}$$
$$\begin{array}{ll}
D'\circ((x\vdash y)\vdash x))
&=(D'(x)\vdash\alpha^k(y))\vdash\alpha^k(x)\\
&+(\alpha^k(x)\vdash D'(y))\vdash\alpha^k(x)+(\alpha^k(x)\vdash\alpha^k(y))\vdash D'(x)
\end{array}$$
for $x, y\in \mathcal{A}.$
\end{definition}

\begin{proposition}\label{pro1}
D is an associative triple derivation of $(\mathcal{A}, \dashv, \vdash,  \alpha)$ if and only
if D is a Jordan triple derivation of $(\mathcal{A}, \dashv, \vdash \alpha)$ such
that $A(x,y,z)+A(y,z,x)+A(z,x,y)=0$ with $A(x,y,z)=(\alpha\circ D)\circ(x\dashv y)\dashv z)$.
\end{proposition}
\begin{proof}
If D is a Jordan triple derivation of $(\mathcal{A}, \dashv, \vdash,  \alpha)$, then $1$ following immediately. $2$ holds because
$$\begin{array}{ll}
& A(x,y,z)+A(y,z,x)+A(z,x,y)=\\
&=\alpha((D(x)\dashv\alpha^k(y))\dashv\alpha^k(z))+(\alpha^k(x)\dashv
D(y))\dashv\alpha^k(z)+(\alpha^k(x)\dashv\alpha^k(y))\dashv D(z))\\
&+\alpha((D(y)\dashv\alpha^k(z))\dashv\alpha^k(x)+(\alpha^k(y))\dashv
D(z))\dashv\alpha^k(x)+(\alpha^k(y)\dashv\alpha^k(z))\dashv D(x))\\
&+\alpha((D(z)\dashv\alpha^k(x))\dashv\alpha^k(y))+(\alpha^k(z)\dashv D(x))\dashv\alpha^k(y)+(\alpha^k(z)\dashv\alpha^k(x))\dashv D(y))\\
&=\alpha((D(x)\dashv\alpha^k(y))\dashv\alpha^k(z))+(\alpha^k(y)\dashv\alpha^k(z))\dashv D(x)+(\alpha^k(z)\dashv D(x))\dashv\alpha^k(y)\\
&+\alpha((\alpha^k(x)\dashv D(y))\dashv\alpha^k(z))+(D(y)\dashv\alpha^k(y))\dashv\alpha^k(x)
+(\alpha^k(z)\dashv\alpha^k(x))\dashv D(y)\\
&+\alpha((\alpha^k(x)\dashv\alpha^k(y))\dashv D(z))+(\alpha^k(y)\dashv D(z))\dashv\alpha^k(x)+(D(z)\dashv\alpha^k(x))\dashv\alpha^k(y)\\
&=0.
\end{array}$$
Therefore, D is a generalized associative triple derivation of Hom-associative $(\mathcal{A}, \dashv, \vdash,  \alpha)$.
\end{proof}

\begin{proposition}
D is an associative triple derivation of $(\mathcal{A}, \dashv, \vdash,  \alpha)$ with respect to associative derivation $\delta$ if
and only if $D$ is Jordan triple derivation of $(A, \dashv, \vdash,  \alpha)$ with respect to a Jordan triple derivation $\delta$ such that
$$
(\alpha^k(x)\dashv\alpha^k(y))\dashv(D-\delta)(z))+(\alpha^k(y)\dashv\alpha^k(z))\dashv(D-\delta)(x))
+(\alpha^k(z)\dashv\alpha^k(x))\dashv(D-\delta)(y))=0,
$$
with $B(x,y,z)=\alpha((\delta(x)\dashv\alpha^k(y)\dashv\alpha^k(z))+(\alpha^k(x)\dashv\delta(y))
\dashv\alpha^k(z))+((\alpha^k(x)\dashv\alpha^k(y))\dashv\delta(z)).$
\end{proposition}
\begin{proof}
If D is a Jordan triple derivation of  $(\mathcal{A}, \dashv, \vdash,  \alpha)$, then $1$ follows immediately. $2$ holds because
$$\begin{array}{ll}
& B(x,y,z)+B(y,z,x)+B(z,x,y)=\\
&=\alpha((\delta(x)\dashv\alpha^k(y))\dashv\alpha^k(z))+(\alpha^k(x)\dashv\delta(y))\dashv\alpha^k(z))+(\alpha^k(x)\dashv
\alpha^k(y))\dashv \delta(z)))\\
&+\alpha((\delta(y)\dashv\alpha^k(z))\dashv\alpha^k(x))+(\alpha^k(y)\dashv\delta(z))\dashv\alpha^k(x))+
(\alpha^k(y)\dashv\alpha^k(z))\dashv \delta(x)))\\
&+\alpha((\delta(z)\dashv\alpha^k(x))\dashv \alpha^k(y))+(\alpha^k(z)\dashv\delta(x))\dashv\alpha^k(y))+(\alpha^k(z)\dashv
\alpha^k(x))\dashv \delta(y)))\\
&=\alpha((\delta(x)\dashv\alpha^k(y))\dashv\alpha^k(z))+(\alpha^k(x)\dashv\delta(y))\dashv\alpha^k(z))
+(\alpha^k(x)\dashv\alpha^k(y))\dashv\delta(z)))\\
&+\alpha((\delta(y)\dashv\alpha^k(z))\dashv\alpha^k(x))+(\alpha^k(y)\dashv\delta(z))\dashv
\alpha^k(x))+(\alpha^k(y)\dashv\alpha^k(z))\dashv\delta(x)))\\
&+\alpha((\delta(z)\dashv\alpha^k(x))\dashv\alpha^k(y))+(\alpha^k(z)\dashv\delta(x))\dashv\alpha^k(y))
+(\alpha^k(z)\dashv\alpha^k(x))\dashv\delta(y)))\\
&=\alpha((\delta(x)\dashv\alpha^k(y))\dashv\alpha^k(z))+(\alpha^k(y)\dashv\alpha^k(z))\dashv\delta(x))+(\alpha^k(z)
\dashv\delta(x))\dashv\alpha^k(y)))\\
&+\alpha((\alpha^k(x)\dashv\delta(y))\dashv\alpha^k(z))+(\delta(y)\dashv\alpha^k(z))\dashv\alpha^k(x))+
(\alpha^k(z)\dashv\alpha^k(x))\dashv\delta(y)))\\
&+\alpha((\alpha^k(x)\dashv\alpha^k(y))\dashv\delta(z))+(\alpha^k(y)\dashv\delta(z))\dashv\alpha^k(x))
+(\delta(z)\dashv\alpha^k(x))\dashv\alpha^k(y)))\\
&=0.
\end{array}$$
Note that $\delta$ is an associative derivation referring to Proposition \ref{pro1}. Therefore, D is a generalized associative triple
derivation of $(\mathcal{A}, \dashv, \vdash,  \alpha)$ with respect to an associative derivation $\delta$ according to Proposition \ref{pro1}.
\end{proof}

\begin{definition}\label{dia2}
An $\alpha$-derivation of the BiHom-associative trialgebra $\mathcal{A}$ is a linear transformation
$d : \mathcal{A} \rightarrow \mathcal{A}$ satisfying
\begin{eqnarray}
\alpha\circ d=d\circ\alpha&,&\\
d(x\dashv y)&=&d(x)\dashv\alpha(y)+\alpha(x)\dashv d(y)\\
d(x\vdash y)&=&d(x)\vdash\alpha(y)+\alpha(x)\vdash d(y),
\end{eqnarray}
for all $x, y\in \mathcal{A}.$
 \end{definition}

\subsection{Derivations of complex Hom-associative dialgebras}
This section illustrates in depth, $\alpha$-derivation of Hom-associative dialgebras in dimension two and three over the field $\mathbb{K}.$ Let
$\left\{e_1,e_2, e_3,\cdots, e_n\right\}$ be a basis of an $n$-dimensional Hom-associative dialgebra $\mathcal{A}.$ The product of basis is denoted by
\begin{eqnarray}
d(e_p)=\sum_{q=1}^nd_{qp}e_q\nonumber.
\end{eqnarray}

We have
\begin{eqnarray}
\sum_{p=1}^nd_{pk}a_{qp}=\sum_{p=1}^na_{pk}d_{qp}& ,&\label{deq1}\\
\sum_{k=1}^n\gamma_{ij}^pd_{rp}=\sum_{k=1}^n\sum_{p=1}^nd_{ki}a_{pj}\gamma_{kp}^q&+&\sum_{k=1}^n\sum_{p=1}a_{ki}d_{pj}\gamma_{kp}^q\label{deq2},\\
\sum_{k=1}^n\delta_{ij}^pd_{rp}=\sum_{k=1}^n\sum_{p=1}^nd_{ki}a_{pj}\delta_{kp}^q&+&\sum_{k=1}^n\sum_{p=1}a_{ki}d_{pj}\delta_{kp}^q\label{deq3}.
\end{eqnarray}

\begin{theorem}
The derivations of 2-dimensional Hom-associative dialgebras have the following form
\end{theorem}
\begin{tabular}{||c||c||c||c||c||c||c||c||c||c||c||c||}
\hline
IC&Der$(d)$ &$Dim(d)$&IC&Der$(d)$&$Dim(d)$\\
			\hline
$\mathcal{TH}_2^{4}$&
$\left(\begin{array}{cccc}
0&0\\
d_{21}&0
\end{array}
\right)$
&
1
&
$\mathcal{TH}_2^{5}$&
$\left(\begin{array}{cccc}
0&0\\
d_{21}&0
\end{array}
\right)$
&
1
\\ \hline
$\mathcal{TH}_2^{6}$&
$\left(\begin{array}{cccc}
0&0\\
d_{21}&0
\end{array}
\right)$
&
1
&
$\mathcal{TH}_2^{8}$&
$\left(\begin{array}{cccc}
0&0\\
d_{21}&0
\end{array}
\right)$
&
1
\\ \hline
\end{tabular}


\begin{theorem}\label{dthieo2}
The derivations of $3$-dimensional Hom-associative dialgebras have the following form
\end{theorem}
\begin{tabular}{||c||c||c||c||c||c||c||c||c||c||c||c||}
\hline
IC&Der$(d)$ &$Dim(d)$&IC&Der$(d)$&$Dim(d)$\\
			\hline
$\mathcal{TH}_3^{1}$&
$\left(\begin{array}{cccc}
0&0&0\\
0&0&0\\
d_{31}&0&d_{33}
\end{array}
\right)$
&
2
&
$\mathcal{TH}_3^{2}$&
$\left(\begin{array}{cccc}
0&0&0\\
0&0&0\\
d_{31}&0&d_{33}
\end{array}
\right)$
&
2
\\ \hline
$\mathcal{TH}_3^{3}$&
$\left(\begin{array}{cccc}
0&0&0\\
d_{21}&0&d_{23}\\
d_{31}&0&d_{33}
\end{array}
\right)$
&
4
&
$\mathcal{TH}_3^{4}$&
$\left(\begin{array}{cccc}
0&0&0\\
0&d_{22}&d_{23}\\
0&d_{32}&d_{33}
\end{array}
\right)$
&
4
\\ \hline
$\mathcal{TH}_3^{5}$&
$\left(\begin{array}{cccc}
0&0&0\\
0&0&0\\
0&d_{32}&d_{33}
\end{array}
\right)$
&
2
&
$\mathcal{TH}_3^{8}$&
$\left(\begin{array}{cccc}
0&0&0\\
0&0&0\\
0&d_{32}&d_{33}
\end{array}
\right)$
&
2
\\ \hline
$\mathcal{TH}_3^{10}$&
$\left(\begin{array}{cccc}
0&0&0\\
d_{21}&0&d_{23}\\
0&0&0
\end{array}
\right)$
&
2
&
$\mathcal{TH}_3^{11}$&
$\left(\begin{array}{cccc}
0&0&0\\
d_{21}&0&d_{23}\\
0&0&0
\end{array}
\right)$
&
2
\\ \hline
$\mathcal{TH}_3^{12}$&
$\left(\begin{array}{cccc}
0&0&0\\
0&0&d_{23}\\
0&0&0
\end{array}
\right)$
&
1
&
$\mathcal{TH}_3^{13}$&
$\left(\begin{array}{cccc}
0&0&0\\
0&0&0\\
0&d_{32}&0
\end{array}
\right)$
&
1
\\ \hline
\end{tabular}
\begin{proof}
Departing from Theorem \ref{dthieo2}, we provide the proof only for one case to illustrate the used approach, the other cases can be addressed similarly with or without
modification(s). Let's consider $\mathcal{TH}_3^{3}$. Applying the systems of equations (\ref{deq1}), (\ref{deq2}) and (\ref{deq3}), we get
$d_{11}=d_{12}=d_{13}=d_{22}=d_{32}=0$. Hence, the derivations of $\mathcal{TH}_3^{3}$ are indicated as follows\\
$d(e_1)=\left(\begin{array}{cccc}
0&0&0\\
1&0&0\\
0&0&0
\end{array}
\right)$,
$d(e_2)=\left(\begin{array}{cccc}
0&0&0\\
0&0&0\\
1&0&0
\end{array}
\right)$,
$d(e_3)=\left(\begin{array}{cccc}
0&0&0\\
0&0&1\\
0&0&0
\end{array}
\right)$,
and
$d(e_4)=\left(\begin{array}{cccc}
0&0&0\\
0&0&0\\
0&0&1
\end{array}
\right)$
is the basis of $Der(Cent(\mathcal{A}))$ and Dim$Der(Cent(\mathcal{A}))=4.$ The derivations of the remaining parts of dimension
two associative dialgebras can be handled in a similar manner as illustrated above.
\end{proof}

\begin{theorem}\label{dthieo3}
The derivations of $4$-dimensional Hom-associative dialgebras have the following form
\end{theorem}
\begin{tabular}{||c||c||c||c||c||c||c||c||c||c||c||c||}
\hline
IC&Der$(d)$ &$Dim(d)$&IC&Der$(d)$&$Dim(d)$\\
			\hline
$\mathcal{TH}_4^{1}$&
$\left(\begin{array}{ccccc}
0&0&0&0\\
d_{21}&0&d_{23}&0\\
0&0&0&0\\
d_{41}&0&d_{43}&0\\
\end{array}
\right)$
&
4
&
$\mathcal{TH}_4^{2}$&
$\left(\begin{array}{cccc}
0&0&0&0\\
d_{21}&0&d_{23}&0\\
0&0&0&0\\
d_{41}&0&d_{43}&0\\
\end{array}
\right)$
&
2
\\ \hline
$\mathcal{TH}_4^{3}$&
$\left(\begin{array}{cccc}
0&0&0&0\\
d_{21}&0&d_{23}&0\\
0&0&0&0\\
d_{41}&0&d_{43}&0\\
\end{array}
\right)$
&
4
&
$\mathcal{TH}_4^{4}$&
$\left(\begin{array}{cccc}
0&0&0&0\\
d_{21}&0&d_{23}&0\\
0&0&0&0\\
d_{41}&0&d_{43}&0\\
\end{array}
\right)$
&
4
\\ \hline
$\mathcal{TH}_4^{5}$&
$\left(\begin{array}{cccc}
0&0&0&0\\
d_{21}&0&d_{23}&0\\
0&0&0&0\\
d_{41}&0&d_{43}&0\\
\end{array}
\right)$
&
4
&
$\mathcal{TH}_4^{6}$&
$\left(\begin{array}{cccc}
0&0&0&0\\
d_{21}&0&d_{23}&0\\
0&0&0&0\\
d_{41}&0&d_{43}&0\\
\end{array}
\right)$
&
4
\\ \hline
$\mathcal{TH}_4^{7}$&
$\left(\begin{array}{cccc}
d_{11}&d_{12}&0&0\\
0&0&0&0\\
0&0&0&d_{34}\\
0&0&0&0\\
\end{array}
\right)$
&
3
&
$\mathcal{TH}_4^{8}$&
$\left(\begin{array}{cccc}
d_{11}&d_{12}&0&0\\
0&0&0&0\\
0&0&0&d_{34}\\
0&0&0&0\\
\end{array}
\right)$
&
3
\\ \hline
$\mathcal{TH}_4^{9}$&
$\left(\begin{array}{cccc}
0&0&0&0\\
0&0&0&0\\
0&0&0&0\\
d_{41}&0&0&0\\
\end{array}
\right)$
&
1
&
$\mathcal{TH}_4^{10}$&
$\left(\begin{array}{cccc}
0&0&0&0\\
0&0&0&0\\
d_{31}&0&0&0\\
0&0&0&0\\
\end{array}
\right)$
&
1
\\ \hline
$\mathcal{TH}_4^{11}$&
$\left(\begin{array}{cccc}
0&0&0&0\\
d_{21}&0&0&0\\
0&0&0&0\\
0&0&d_{43}&0\\
\end{array}
\right)$
&
1
&
$\mathcal{TH}_4^{12}$&
$\left(\begin{array}{cccc}
0&0&0&0\\
0&0&0&0\\
d_{31}&0&0&0\\
0&0&0&0\\
\end{array}
\right)$
&
1
\\ \hline
$\mathcal{TH}_4^{13}$&
$\left(\begin{array}{cccc}
0&0&0&0\\
0&0&0&0\\
0&0&0&0\\
d_{41}&0&0&0\\
\end{array}
\right)$
&
1
&
$\mathcal{TH}_4^{14}$&
$\left(\begin{array}{cccc}
0&0&0&0\\
0&0&0&0\\
0&0&0&0\\
d_{41}&0&0&0\\
\end{array}
\right)$
&
1
\\ \hline
$\mathcal{TH}_4^{15}$&
$\left(\begin{array}{cccc}
0&0&0&0\\
0&0&0&0\\
0&0&0&0\\
d_{41}&0&0&0\\
\end{array}
\right)$
&
1
&
$\mathcal{TH}_4^{16}$&
$\left(\begin{array}{cccc}
0&0&0&0\\
0&0&0&0\\
0&0&0&0\\
d_{41}&0&0&0\\
\end{array}
\right)$
&
1
\\ \hline	
$\mathcal{TH}_4^{17}$&
$\left(\begin{array}{cccc}
0&0&0&0\\
0&0&0&0\\
0&0&d_{23}&0\\
d_{41}&0&0&0\\
\end{array}
\right)$
&
2
&
$\mathcal{TH}_4^{18}$&
$\left(\begin{array}{cccc}
0&0&0&0\\
0&0&d_{23}&d_{24}\\
0&0&-d_{23}&-d_{24}\\
0&0&d_{23}&d_{24}\\
\end{array}
\right)$
&
2
\\ \hline		
\end{tabular}

\begin{tabular}{||c||c||c||c||c||c||c||c||c||c||c||c||}
\hline
IC&Der$(d)$ &$Dim(d)$&IC&Der$(d)$&$Dim(d)$\\
			\hline
$\mathcal{TH}_4^{19}$&
$\left(\begin{array}{cccc}
0&0&0&0\\
0&0&d_{23}&d_{24}\\
0&0&0&0\\
0&0&0&0\\
\end{array}
\right)$
&
2
&
$\mathcal{TH}_4^{20}$&
$\left(\begin{array}{cccc}
0&0&0&0\\
0&0&d_{23}&d_{24}\\
0&0&0&0\\
0&0&0&0\\
\end{array}
\right)$
&
2
\\ \hline
\end{tabular}

\begin{proof}
Departing from Theorem \ref{dthieo3}, we provide the proof only for one case to illustrate the used approach, the other cases can be addressed similarly with or without
modification(s). Let's consider $\mathcal{TH}_4^{1}$. Applying the systems of equations (\ref{deq1}), (\ref{deq2}) and (\ref{deq3}), we get
$d_{11}=d_{12}=d_{13}=d_{14}=d_{22}=d_{24}=d_{31}=d_{32}=d_{33}=d_{34}=d_{42}=d_{44}=0$. Hence, the derivations of $\mathcal{TH}_4^{1}$ are indicated as follows\\
$d(e_1)=\left(\begin{array}{ccccc}
0&0&0&0\\
1&0&0&0\\
0&0&0&0\\
0&0&0&0
\end{array}
\right)$,
$d(e_2)=\left(\begin{array}{cccc}
0&0&0&0\\
0&0&0&0\\
0&0&0&0\\
1&0&0&0
\end{array}
\right)$,
$d(e_3)=\left(\begin{array}{cccc}
0&0&0&0\\
0&1&0&0\\
0&0&0&0\\
0&0&0&0
\end{array}
\right)$
and
$d(e_4)=\left(\begin{array}{cccc}
0&0&0&0\\
0&0&0&0\\
0&0&0&0\\
0&0&1&0
\end{array}
\right)$
is the basis of $Der(Cent(\mathcal{A}))$ and Dim$Der(Cent(\mathcal{A}))=4.$ The derivations of the remaining parts of dimension
two associative dialgebras can be handled in a similar manner as illustrated above.
\end{proof}

\begin{corollary}\,
\begin{itemize}
	\item The dimensions of the derivations of two-dimensional associative trialgebras range between zero and two.
	\item The dimensions of the derivations of three-dimensional associative trialgebras range between zero and three.
	\item The dimensions of the derivations of four-dimensional associative trialgebras range between one  and four.
\end{itemize}
\end{corollary}

 \section{Centroids of Complex Hom-associative dialgebras}
\subsection{Properties of centroids Hom-associative dialgebras}
In this section, we set forward the following results on properties of centroids of Hom-associative dialgebras $\mathcal{A}$.
\begin{definition}
Let $(\mathcal{A}, \dashv, \vdash,  \alpha)$ be a  Hom-associative dialgebra. A linear map
 $\psi : \mathcal{A}\rightarrow \mathcal{A}$ is called an element of $(\alpha)$-element of centroids on $\mathcal{A}$ if, for all $x, y\in \mathcal{A}$,
\begin{eqnarray}
\alpha\circ\psi&=&\psi\circ\alpha,\\
\psi(x)\dashv \alpha(y)&=&\psi(x)\dashv\psi(y)=\alpha(x)\dashv \psi(y),\\
 \psi(x)\vdash \alpha(y)&=&\psi(x)\vdash \psi(y)=\alpha(x)\vdash \psi(y).
\end{eqnarray}
The set of all  elements of $(\alpha)$-centroid of $\mathcal{A}$ is denoted $Cent_{(\alpha)}(\mathcal{A})$.
The centroid of $\mathcal{A}$ is denoted $Cent(\mathcal{A})$.
 \end{definition}

\begin{definition}
Let $\mathcal{H}$ be a nonempty subset of $\mathcal{A}$. The subset
\begin{equation}
Z_{\mathcal{A}}(\mathcal{H})=\left\{x\in\mathcal{H} | \alpha(x)\bullet \mathcal{H} = \mathcal{H}\bullet\alpha(x)=0\right\},
\end{equation}
is said to be centralizer of $\mathcal{H}$ in $\mathcal{A}$, where $\bullet$ is $\dashv$ and $\vdash$, respectively.
\end{definition}

\begin{definition}
Let $\psi\in End(\mathcal{A})$. If $\psi(\mathcal{A})\subseteq Z(\mathcal{A})$ and $\psi(\mathcal{A}^2)=0$, then $\psi$ is called a central derivation.
The set of all central derivations of $\mathcal{A}$ is  denoted by $\mathcal{C}(\mathcal{A})$.
\end{definition}

\begin{proposition}
Consider $(\mathcal{A}, \dashv, \vdash, \alpha)$ a Hom-associative dialgebra. Then,
\begin{enumerate}
	\item [i)]$\Gamma(\mathcal{A})Der(\mathcal{T})\subseteq Der(\mathcal{A})$.
		\item [ii)]$\left[\Gamma(\mathcal{A}), Dr(\mathcal{A})\right]\subseteq\Gamma(\mathcal{A}).$
	\item [iii)]$\left[\Gamma(\mathcal{A}), \Gamma(\mathcal{A})\right](\mathcal{A})\subseteq \Gamma(\mathcal{A})$ and $\left[\Gamma(\mathcal{A}), \Gamma(\mathcal{A})\right](\mathcal{A}^2)=0.$
\end{enumerate}
 \end{proposition}
\begin{proof}
The proof
 of parts $i)-iii)$ is straightforward by using definitions of derivations and centroids.
\end{proof}


\begin{proposition}
Let $(\mathcal{A}, \dashv, \vdash \alpha)$ be a Hom-associative dialgebra and $\varphi\in Cent(\mathcal{A}),\, d\in Der(\mathcal{A}).$
Then, $\varphi\circ d$ is an $\alpha$-derivation of $\mathcal{A}.$
\end{proposition}
\begin{proof}
Indeed, if $x, y\in \mathcal{A}$, then
$$\begin{array}{ll}
(\varphi\circ d)(x\bullet y)
&= \varphi(d(x)\bullet\alpha(y)+\alpha(x)\bullet d(y))\\
&= \varphi(d(x)\bullet y)+\varphi(x\bullet d(y))=(\varphi\circ d)(x)\bullet\alpha(y)+\alpha(x)\bullet(\varphi\circ d)(y),
\end{array}$$
where $\bullet$ is $\dashv$ and $\vdash$, respectively.
\end{proof}

\begin{proposition}
Let $(\mathcal{A}, \dashv, \vdash, \alpha)$ be a Hom-associative dialgebra over a field $\mathbb{F}$. Hence, $\mathcal{C}(\mathcal{A})=Cent(\mathcal{A})\cap Der(\mathcal{A}).$
\end{proposition}

\begin{proof}
If $\psi\in Cent(\mathcal{A})\cap Der(\mathcal{A})$, then grounded on definition of $Cent(\mathcal{A}$ and $Der(\mathcal{A})$, we have

$\psi(x\bullet y)=\psi(x)\bullet\alpha(y)+\alpha(x)\bullet\psi(y)$ and $\psi(x\bullet y)=\psi(x)\circ\alpha(y)=\alpha(x)\circ\psi(y)$ for $x,y\in \mathcal{A}.$
The yields $\psi(\mathcal{A}\mathcal{A})=0$ and $\psi(\mathcal{A})\subseteq  Z(\mathcal{A})$ i.e $Cent(\mathcal{A})\cap Der(\mathcal{A})\subseteq Cent(\mathcal{A}).$
The inverse is obvious since $\mathcal{C}(\mathcal{A})$ is in both $Cent(\mathcal{A})$ and $Der(\mathcal{A}),$ where $\bullet$ is $\dashv$ and $\vdash$, respectively.
\end{proof}

\begin{proposition}
Let $(\mathcal{A}, \dashv, \vdash, \alpha)$ be a Hom-associative dialgebra. Therefore,  for any $d\in Der(\mathcal{A})$ and $\varphi\in Cent(\mathcal{A})$, we have
\begin{enumerate}
	\item [(i)] The composition $d\circ\varphi$ is in $Cent(\mathcal{A})$, if and only if $\varphi\circ d$ is a central $\alpha$-derivation of $\mathcal{A}.$
		\item [(ii)] The  composition $d\circ\varphi$ is a $\alpha$-derivation of $\mathcal{A}$, if and only if $\left[d,\varphi\right]$ is a central $\alpha$-derivation of $\mathcal{A}.$
\end{enumerate}
\end{proposition}

\begin{proof}
\begin{enumerate}
	\item [i)]For any $\varphi\in Cent(\mathcal{A}),\, d\in Der(\mathcal{A}),\, \forall\,x,y\in \mathcal{A}$, we have
	$$\begin{array}{ll}
d\circ\varphi(x\bullet y)=d\circ\varphi(x)\bullet y
&=d\circ\varphi(x)\bullet y+\varphi(x)\bullet d(y)\\
&=d\circ\varphi(x)\bullet y+\varphi\circ d(x\bullet y)-\varphi\circ d(x)\bullet y.
\end{array}$$
Thus, $(d\circ\varphi-\varphi\circ d)(x\bullet y)=(d\bullet\varphi-\varphi\circ d)(x)\bullet y.$
	\item [ii)] Let $d\circ\varphi\in Der(\mathcal{A})$. Using $\left[d,\varphi\right]\in Cent(\mathcal{A})$, we get
	\begin{equation}\label{eq1}
	\left[d,\varphi\right](x\bullet y)=(\left[d, \varphi\right](x))\bullet\alpha(y)=\alpha(x)\bullet(\left[d,\varphi\right](y))
	\end{equation}
	On the other side, $\left[d, \varphi\right]d\circ\varphi-\varphi\circ d$ and $d\circ\varphi, \varphi\circ d\in Der(\mathcal{A}).$ Therefore,	
	\begin{equation}\label{eq2}
\left[d, \varphi\right](x\bullet y)=(d(\varphi\circ(x))\bullet\alpha(y)+\alpha(x)\bullet(d\circ\varphi(y))-(\varphi\circ d(x))\bullet\alpha(y)-\alpha(x)\bullet(\varphi\circ d(y)).
\end{equation}
Referring to (\ref{eq1}) and (\ref{eq2}), we get $\alpha(x)\bullet(\left[d, \varphi\right])(y)=(\left[d, \varphi\right])(x)\bullet\alpha(y)=0.$

\noindent At this stage of analysis,let $\left[d, \varphi\right]$ be a central $\alpha$-derivation of $\mathcal{A}$. Then,
$$\begin{array}{ll}
d\circ\varphi(x\bullet y)
&=\left[d\circ\varphi\right](x\bullet y)+(\varphi\circ d)(x\bullet y)\\
&=\varphi(\circ d(x)\bullet\alpha(y))+\varphi(\alpha(x)\bullet d(y))\\
&=(\varphi\circ d)(x)\bullet\alpha(y)+\alpha(x)\bullet(\varphi\circ d)(y),
\end{array}$$
\end{enumerate}
where $\bullet$ represents the products $\dashv$ and $\vdash$, respectively.
\end{proof}

\subsection{Centroids of complex Hom-associative dialgebras}
This section elaborates the details of $\alpha$-centroids of Hom-associative dialgebras in dimension two and three over the field $\mathbb{K}.$ Let
$\left\{e_1,e_2, e_3,\cdots, e_n\right\}$ be a basis of an $n$-dimensional Hom-associative dialgebra $\mathcal{A}.$ The product of basis
\begin{eqnarray}
\psi(e_p)=\sum_{q=1}^nc_{qp}e_q\nonumber.
\end{eqnarray}

\begin{eqnarray}
\sum_{p=1}^nc_{pi}a_{qp}&=&\sum_{p=1}^na_{pi}c_{qp},\label{Ceq2}\\
\sum_{k=1}^n\sum_{p=1}^nc_{ki}a_{pj}\gamma_{kp}^q&=&\sum_{k=1}^n\sum_{p=1}^nc_{ki}c_{pj}\gamma_{kp}^q=\sum_{k=1}^n\sum_{p=1}^na_{ki}c_{pj}\gamma_{kp}^q,\label{Ceq3}\\
\sum_{k=1}^n\sum_{p=1}^nc_{ki}a_{pj}\delta_{kp}^q&=&\sum_{k=1}^n\sum_{p=1}^nc_{ki}c_{pj}\delta_{kp}^q=\sum_{k=1}^n\sum_{p=1}^na_{ki}c_{pj}\delta_{kp}^q\label{Ceq4}.
\end{eqnarray}

\begin{theorem}\label{Cthieo1}
The centroids of $2$-dimensional Hom-associative dialgebras have the following form
\end{theorem}
\begin{tabular}{||c||c||c||c||c||c||c||c||c||c||c||c||}
\hline
IC&$Cent(\mathcal{A})$ &$Dim(Cent(\mathcal{A}))$&IC&$Cent(\mathcal{A})$&$Dim(Cent(\mathcal{A}))$\\
			\hline
$\mathcal{TH}_2^{1}$&
$\left(\begin{array}{cccc}
c_{11}&0\\
0&c_{22}
\end{array}
\right)$
&
2
&
$\mathcal{TH}_2^{2}$&
$\left(\begin{array}{cccc}
0&0\\
c_{21}&0
\end{array}
\right)$
&
1
\\ \hline
$\mathcal{TH}_2^{3}$&
$\left(\begin{array}{cccc}
-c_{11}&0\\
c_{21}&c_{22}
\end{array}
\right)$
&
3
&
$\mathcal{TH}_2^{4}$&
$\left(\begin{array}{cccc}
c_{11}&0\\
c_{21}&c_{22}
\end{array}
\right)$
&
3
\\ \hline
$\mathcal{TH}_2^{5}$&
$\left(\begin{array}{cccc}
0&0\\
c_{21}&0
\end{array}
\right)$
&
1
&
$\mathcal{TH}_2^{6}$&
$\left(\begin{array}{cccc}
c_{11}&0\\
0&c_{22}
\end{array}
\right)$
&
2
\\ \hline
$\mathcal{TH}_2^{7}$&
$\left(\begin{array}{cccc}
c_{11}&0\\
0&c_{22}
\end{array}
\right)$
&
2
&
$\mathcal{TH}_2^{8}$&
$\left(\begin{array}{cccc}
0&0\\
c_{21}&0
\end{array}
\right)$
&
1
\\ \hline
$\mathcal{TH}_2^{9}$&
$\left(\begin{array}{cccc}
c_{11}&0\\
c_{21}&c_{22}
\end{array}
\right)$
&
3
&
&
&
\\ \hline
\end{tabular}

\begin{proof}
Departing from Theorem \ref{Cthieo1}, we provide the proof only for one case to illustrate the used approach, the other cases can be addressed similarly with or without
modification(s). Let's consider $\mathcal{TH}_2^{1}$. Applying the systems of equations (\ref{Ceq2}), (\ref{Ceq3}) and (\ref{Ceq4}), we get
$c_{12}=c_{21}=0$. Hence, the centroids of $\mathcal{TH}_2^{1}$ are indicated as follows\\
$c(e_1)=\left(\begin{array}{cccc}
1&0\\
0&0
\end{array}
\right)$ and $c(e_2)=\left(\begin{array}{cccc}
0&0\\
0&1
\end{array}
\right)$
is the basis of $Der(Cent(\mathcal{A}))$ and  Dim$Der(Cent(\mathcal{A}))=2.$ The centroids of the remaining parts of dimension two associative dialgebras can be handled in a
similar manner as illustrated above.
\end{proof}

\begin{theorem}\label{Cthieo2}
The centroids of $3$-dimensional Hom-associative dialgebras have the following form
\end{theorem}
\begin{tabular}{||c||c||c||c||c||c||c||c||c||c||c||c||}
\hline
IC&$Cent(\mathcal{A})$ &$Dim(Cent(\mathcal{A}))$&IC&$Cent(\mathcal{A})$&$Dim(Cent(\mathcal{A}))$\\
			\hline
$\mathcal{TH}_3^{1}$&
$\left(\begin{array}{cccc}
0&0&0\\
c_{21}&0&0\\
c_{31}&0&0
\end{array}
\right)$
&
2
&
$\mathcal{TH}_3^{2}$&
$\left(\begin{array}{cccc}
0&0&0\\
c_{21}&0&0\\
c_{31}&0&0
\end{array}
\right)$
&
2
\\ \hline
$\mathcal{TH}_3^{3}$&
$\left(\begin{array}{cccc}
c_{11}&0&0\\
c_{21}&c_{11}&c_{11}\\
c_{31}&0&0
\end{array}
\right)$
&
3
&
$\mathcal{TH}_3^{4}$&
$\left(\begin{array}{cccc}
c_{11}&0&0\\
0&0&0\\
0&0&0
\end{array}
\right)$
&
1
\\ \hline
$\mathcal{TH}_3^{5}$&
$\left(\begin{array}{cccc}
0&0&0\\
0&0&c_{23}\\
0&0&c_{33}
\end{array}
\right)$
&
2
&
$\mathcal{TH}_3^{6}$&
$\left(\begin{array}{cccc}
0&0&0\\
\frac{1}{2}c_{21}&0&0\\
c_{31}&0&0
\end{array}
\right)$
&
2
\\ \hline
$\mathcal{TH}_3^{7}$&
$\left(\begin{array}{cccc}
c_{11}&0&0\\
0&0&0\\
0&0&0
\end{array}
\right)$
&
1
&
$\mathcal{TH}_3^{8}$&
$\left(\begin{array}{cccc}
c_{11}&0&0\\
0&0&c_{23}\\
0&0&c_{33}
\end{array}
\right)$
&
3
\\ \hline
$\mathcal{TH}_3^{9}$&
$\left(\begin{array}{cccc}
c_{11}&0&0\\
0&0&0\\
0&0&0
\end{array}
\right)$
&
1
&
$\mathcal{TH}_3^{10}$&
$\left(\begin{array}{cccc}
0&0&0\\
c_{21}&0&c_{23}\\
c_{31}&0&c_{33}
\end{array}
\right)$
&
4
\\ \hline
$\mathcal{TH}_3^{11}$&
$\left(\begin{array}{cccc}
0&0&0\\
c_{21}&0&0\\
c_{31}&0&0
\end{array}
\right)$
&
2
&
$\mathcal{TH}_3^{12}$&
$\left(\begin{array}{cccc}
0&0&0\\
c_{21}&0&0\\
c_{31}&0&0
\end{array}
\right)$
&
2
\\ \hline
$\mathcal{TH}_3^{13}$&
$\left(\begin{array}{cccc}
c_{11}&0&0\\
0&c_{22}&0\\
0&0&c_{33}
\end{array}
\right)$
&
3
&
$\mathcal{TH}_3^{14}$&
$\left(\begin{array}{cccc}
0&0&0\\
0&c_{22}&0\\
0&0&0
\end{array}
\right)$
&
1
\\ \hline
\end{tabular}

\begin{proof}
Departing from Theorem \ref{Cthieo2}, we provide the proof only for one case to illustrate the used approach, the other cases can be addressed similarly with or without
modification(s). Let's consider $\mathcal{TH}_3^{3}$. Applying the systems of equations (\ref{Ceq2}), (\ref{Ceq3}) and (\ref{Ceq4}), we get
$c_{12}=c_{13}=c_{32}=c_{33}=0$ and $c_{22}=c_{23}=c_{11}$. Hence, the centroids of $\mathcal{TH}_3^{3}$ are indicated as follows\\
$c(e_1)=\left(\begin{array}{cccc}
1&0&0\\
0&1&1\\
0&0&0
\end{array}
\right)$,\quad
$c(e_2)=\left(\begin{array}{cccc}
0&0&0\\
1&0&0\\
0&0&0
\end{array}
\right)$ and
$c(e_3)=\left(\begin{array}{cccc}
0&0&0\\
0&0&0\\
1&0&0
\end{array}
\right)$
is the basis of $Der(c)$ and  Dim$Der(c)=3.$ The centroids of the remaining parts of dimension three associative dialgebras can be handled in a
similar manner as illustrated above.
\end{proof}

\begin{theorem}\label{Cthieo3}
The centroids of $4$-dimensional Hom-associative dialgebras have the following form
\end{theorem}

\begin{tabular}{||c||c||c||c||c||c||c||c||c||c||c||c||}
\hline
IC&$Cent(\mathcal{A})$ &$Dim$&IC&$Cent(\mathcal{A})$&$Dim$\\
			\hline
$\mathcal{TH}_4^{1}$&
$\left(\begin{array}{ccccc}
0&0&0&0\\
c_{21}&0&c_{23}&0\\
0&0&0&0\\
c_{41}&0&c_{43}&0\\
\end{array}
\right)$
&
4
&
$\mathcal{TH}_4^{2}$&
$\left(\begin{array}{cccc}
0&0&0&0\\
c_{21}&0&c_{23}&0\\
c_{31}&0&c_{33}&0\\
c_{41}&0&c_{43}&0\\
\end{array}
\right)$
&
6
\\ \hline
$\mathcal{TH}_4^{3}$&
$\left(\begin{array}{cccc}
0&0&0&0\\
d_{21}&0&d_{23}&0\\
c_{31}&0&c_{33}&0\\
d_{41}&0&d_{43}&0\\
\end{array}
\right)$
&
6
&
$\mathcal{TH}_4^{4}$&
$\left(\begin{array}{cccc}
0&0&0&0\\
c_{21}&0&0&0\\
0&0&0&0\\
0&0&c_{43}&0\\
\end{array}
\right)$
&
2
\\ \hline
\end{tabular}

\begin{tabular}{||c||c||c||c||c||c||c||c||c||c||c||c||}
\hline
IC&$Cent(\mathcal{A})$ &$Dim$&IC&$Cent(\mathcal{A})$&$Dim$\\
			\hline
$\mathcal{TH}_4^{5}$&
$\left(\begin{array}{cccc}
0&0&0&0\\
c_{21}&0&c_{23}&0\\
0&0&0&0\\
c_{41}&0&c_{43}&0\\
\end{array}
\right)$
&
4
&
$\mathcal{TH}_4^{6}$&
$\left(\begin{array}{cccc}
0&0&0&0\\
c_{21}&0&c_{23}&0\\
0&0&0&0\\
c_{41}&0&c_{43}&0\\
\end{array}
\right)$
&
4
\\ \hline
$\mathcal{TH}_4^{7}$&
$\left(\begin{array}{cccc}
c_{11}&c_{12}&0&0\\
c_{21}&c_{22}&0&0\\
0&0&0&0\\
c_{41}&0&c_{43}&0\\
\end{array}
\right)$
&
4
&
$\mathcal{TH}_4^{8}$&
$\left(\begin{array}{cccc}
0&c_{12}&0&0\\
0&c_{22}&0&0\\
0&0&0&0\\
0&0&0&0\\
\end{array}
\right)$
&
2
\\ \hline
$\mathcal{TH}_4^{9}$&
$\left(\begin{array}{cccc}
0&0&0&0\\
0&c_{22}&0&0\\
0&0&c_{33}&0\\
0&0&0&0\\
\end{array}
\right)$
&
2
&
$\mathcal{TH}_4^{10}$&
$\left(\begin{array}{cccc}
c_{11}&0&0&0\\
0&0&0&0\\
c_{31}&0&0&0\\
0&0&0&0\\
\end{array}
\right)$
&
2
\\ \hline
$\mathcal{TH}_4^{11}$&
$\left(\begin{array}{cccc}
c_{11}&0&0&0\\
0&0&0&0\\
c_{31}&0&c_{33}&0\\
0&0&c_{43}&0\\
\end{array}
\right)$
&
4
&
$\mathcal{TH}_4^{12}$&
$\left(\begin{array}{cccc}
c_{11}&0&0&0\\
0&0&0&0\\
c_{31}&0&0&0\\
0&0&0&0\\
\end{array}
\right)$
&
2
\\ \hline
$\mathcal{TH}_4^{13}$&
$\left(\begin{array}{cccc}
c_{11}&0&0&0\\
0&0&0&0\\
0&0&0&0\\
c_{41}&0&0&c_{44}\\
\end{array}
\right)$
&
3
&
$\mathcal{TH}_4^{14}$&
$\left(\begin{array}{cccc}
c_{11}&0&0&0\\
0&0&0&0\\
0&0&0&0\\
c_{41}&0&0&0\\
\end{array}
\right)$
&
2
\\ \hline			
$\mathcal{TH}_4^{15}$&
$\left(\begin{array}{cccc}
0&0&0&0\\
0&c_{22}&0&0\\
0&c_{32}&0&0\\
0&0&0&0\\
\end{array}
\right)$
&
2
&
$\mathcal{TH}_4^{16}$&
$\left(\begin{array}{cccc}
c_{11}&0&0&0\\
0&c_{22}&0&0\\
0&c_{32}&0&0\\
c_{41}&0&0&0\\
\end{array}
\right)$
&
4
\\ \hline
$\mathcal{TH}_4^{17}$&
$\left(\begin{array}{cccc}
c_{11}&0&0&0\\
0&c_{22}&0&0\\
0&c_{32}&0&0\\
c_{41}&0&0&0\\
\end{array}
\right)$
&
4
&
$\mathcal{TH}_4^{18}$&
$\left(\begin{array}{cccc}
0&0&0&0\\
\frac{1-\sqrt{1-4c_{21}^2}}{2}&0&0&c_{24}\\
0&0&0&0\\
0&0&0&0
\end{array}
\right)$
&
2
\\ \hline
\end{tabular}

\begin{proof}
Departing from Theorem \ref{Cthieo3}, we provide the proof only for one case to illustrate the used approach, the other cases can be addressed similarly with or without
modification(s). Let's consider $\mathcal{TH}_4^{1}$. Applying the systems of equations (\ref{Ceq2}), (\ref{Ceq3}) and (\ref{Ceq4}), we get
$c_{11}=c_{12}=c_{13}=c_{14}=0=c_{22}=c_{24}=c_{31}=c_{32}=c_{33}=c_{34}=c_{42}=c_{44}=0$. Hence, the centroids of $\mathcal{TH}_4^{1}$ are indicated as follows\\
$c(e_1)=\left(\begin{array}{ccccc}
0&0&0&0\\
1&0&0&0\\
0&0&0&0\\
0&0&0&0
\end{array}
\right)$,
$c(e_2)=\left(\begin{array}{cccc}
0&0&0&0\\
0&0&0&0\\
0&0&0&0\\
1&0&0&0
\end{array}
\right)$,
$c(e_3)=\left(\begin{array}{cccc}
0&0&0&0\\
0&1&0&0\\
0&0&0&0\\
0&0&0&0
\end{array}
\right)$
and
$c(e_4)=\left(\begin{array}{cccc}
0&0&0&0\\
0&0&0&0\\
0&0&0&0\\
0&0&1&0
\end{array}
\right)$
is the basis of $Der(Cent(\mathcal{A}))$ and Dim$Der(Cent(\mathcal{A}))=4.$ The derivations of the remaining parts of dimension
two associative dialgebras can be handled in a similar manner as illustrated above.
\end{proof}

\begin{corollary}\,
\begin{itemize}
	\item The dimensions of the centroids of two-dimensional associative trialgebras range between one and three.
	\item The dimensions of the centroids of three-dimensional associative trialgebras range between one and five.
	\item The dimensions of the centroids of four-dimensional associative trialgebras range between zero and six.
\end{itemize}
\end{corollary}


\begin{thebibliography}{9}
\bibitem{AZAB2}
A.~ Makhlouf and Ahmed Zahari,  \emph{Structure and Classification of Hom-Associative Algebras}, Acta et commentationes universitis Tartuensis de mathematica, vol $24 (1),  (2020), 79-102$.

\bibitem{AI}
A. Zahari and I. Bakayoko,   \emph{On BiHom-Associative dialgebras}, Open J. Math. Sci. vol(7), $(2023), 96-117$.


\bibitem{M2}
A.~ Makhlouf, \emph{Alg\`ebres associatives et calcul formel}, Theoret. Comput. Sci. $187 (1997)$, no. $1$-$2$, $123$-$145$.

\bibitem{LFG} J.~L.~Loday, A.~Frabetti, F.~Chapoton and F.~ Gouchot \emph{Dialgebras and Related operads}, Lecture Notes sur Math. Berlin :
Springer 2001.

\bibitem{LS} D.~ Larsson  and S.~ Silvestrov, \emph{Quasi-hom-Lie algebras, central extensions and 2-cocycle-
like identities}, J. Algebra 288 (2005), 321--344.

\bibitem{MS}
A.~  Makhlouf and S.~ Silvestrov, \emph{Hom-algebra structures}, J. Gen. Lie Theory Appl. Vol.$2$ (2008), No.$2,51$-$64$.

\bibitem{WR} W.~ Basri and Rokhsiboev \emph{On low dimensional diassociative algebras, Proceedings of Third conference on Research and Education in Mathematics},(ICREM3), UPM, Malaysia : (2007) 164--170.

\bibitem{Fregier} Y.~ Fregier, A.~ Gohr and S.~ Silvestrov, \emph{Unital algebras of Hom-associative type and surjective
or injective twistings}, J. Gen. Lie Theory Appl. Vol. 3 (4), (2009) 285--295.

\bibitem{Dim4Zhang}
A.~ Armour, H.~ Chen and Y.~ Zhang, \emph{Classification of $4$-dimensional superalgebras}, Comm. in Algebra $37(2009), 3697$-$3728.$

\bibitem{MRW} M.~ Ikrom, Rokhsiboev, S.~ Isamiddin, Rakhimov and W.~ Basri \emph{Classification of $3$-dimensional Complex diassociative algebras}, Malaysian Journal of Mathematical Sciences 4(2) : (2010) 241--254.

\bibitem{RSW} I.~M.~Rikhsiboev, I.~S.~ Rakhimov and W.~Basri \emph{Diassocitive algebras and their derivations}, J. phys : Conf. Ser 553
(2014)012006.

\bibitem{Li}
X.~ Li, \emph{Structures of multiplicative Hom-Lie algebras}, Advances in Mathematics (China), $43(6) (2014) 817$-$823$.


\bibitem{WRR}W.~ Basri,I.~S.~Rakhimov and I.~M.~Rokhsiboev, \emph{Four-Dimension Nilpotent Diassociative algebras}, J. Generalised Lie
Theory Appl. $28.$ doi : $10.4172/1736-4337.1000218.$


\bibitem{Chen}
X.~ Chen and W.~ Han, \emph{Classification of multiplicative simple Hom-Lie algebras}, J. Lie Theory $21 (4) (2015)$

\bibitem{IR} S.~ Isamiddin  and Rakhimov, \emph{On central Extensions of Associative Dialgebras}, J. Physics : conf. Ser. 697 (2016).

\bibitem{AZB}
Ahmed Zahari, \emph{Classification of $3$-dimensional BiHom-Associative and BiHom-Bialgebras}, arXiv : 1907.0080 v1
 [math.RA] 28 Jun 2019.
\end{thebibliography}
\end{document}